\UseRawInputEncoding
\documentclass[11pt,twoside,a4paper]{amsart}
\usepackage{amssymb}
\usepackage{amsmath,amscd}
\usepackage[initials,nobysame,shortalphabetic]{amsrefs}
\usepackage[all,cmtip]{xy}
\usepackage{hyperref}
\date{\today}
\usepackage{geometry}

\usepackage{xcolor}

\def\la{\langle}
\def\ra{\rangle}

\def\w{\wedge}

\def\dbar{\bar\partial}

\def\R{{\mathbf R}}
\def\N{{\mathbf N}}
\def\C{{\mathbf C}}
\def\P{{\mathbf P}}
\def\w{{\wedge}}
\def\Pk{{\mathbf P}}

\def\D{{\mathcal D}}

\def\F{{\mathcal F}}

\def\codim{{\rm codim\,}}

\def\E{{\mathcal E}}

\def\O{{\mathcal O}}

\def\U{{\mathcal U}}

\def\nbh{neighborhood }

\def\be{\begin{equation}}
\def\ee{\end{equation}}

\def\Ok{\mathcal O}

\def\1{{\bf 1}}

\def\div{{\rm div}}
\def\Pk{{\mathbb P}}

\DeclareMathOperator{\supp}{supp}
\DeclareMathOperator{\psh}{PSH}
\DeclareMathOperator{\MA}{MA}

\def\B{{\mathbf B}} 
\def\G{{\mathcal G}}
\def\Di{{\mathbf D}}
\def\F{{\mathcal F}}
\def\S{{S}}

\newcommand{\ma}{Monge-Amp\`ere }

\definecolor{darkred}{rgb}{.77, .01,.2}

\mathcode`l="8000
\begingroup
\makeatletter
\lccode`\~=`\l
\DeclareMathSymbol{\lsb@l}{\mathalpha}{letters}{`l}
\lowercase{\gdef~{\ifnum\the\mathgroup=\m@ne \ell \else \lsb@l \fi}}%
\makeatother
\endgroup

\newtheorem{thm}{Theorem}[section]
\newtheorem{lma}[thm]{Lemma}
\newtheorem{cor}[thm]{Corollary}
\newtheorem{prop}[thm]{Proposition}

\theoremstyle{definition}

\newtheorem{df}[thm]{Definition}

\theoremstyle{remark}

\newtheorem{preremark}[thm]{Remark}
\newtheorem{preex}[thm]{Example}

\newenvironment{remark}{\begin{preremark}}{\qed\end{preremark}}
\newenvironment{ex}{\begin{preex}}{\qed\end{preex}}

\numberwithin{equation}{section}

\title[Non-pluripolar energy and the complex Monge-Amp\`ere operator]
{Non-pluripolar energy and the complex Monge-Amp\`ere operator}

\begin{document}
\date{\today}

\author{Mats Andersson \& David Witt Nystr\"om \& Elizabeth  Wulcan}
\address{Department of Mathematical Sciences\\Chalmers University of Technology and the University of Gothenburg\\412 96 Gothenburg\\SWEDEN}
\email{matsa@chalmers.se, wittnyst@chalmers.se, wulcan@chalmers.se}

\thanks{The authors were partially supported by the
  Swedish Research Council.
  The second author was also partially supported by a grant from the
  Knut and Alice Wallenberg Foundation and the 
  third author was also partially supported by the G\"oran
Gustafsson Foundation for Research in Natural Sciences and Medicine.}

\begin{abstract}
Given a domain $\Omega\subset \C^n$ we introduce a class of
plurisubharmonic (psh) functions $\G(\Omega)$ and Monge-Amp\`ere
operators $u\mapsto [dd^c u]^p$, $p\leq n$, on $\G(\Omega)$ that
extend the Bedford-Taylor-Demailly Monge-Amp\`ere operators.
Here $[dd^c u]^p$ is a closed positive current of bidegree
$(p,p)$ that dominates
the non-pluripolar Monge-Amp\`ere current $\la dd^c u\ra^p$. 
We prove that $[dd^c u]^p$
is the limit of Monge-Amp\`ere currents of
certain natural regularizations of $u$.

On a compact K\"ahler manifold $(X, \omega)$ we introduce a notion of
non-pluripolar energy 
and a corresponding finite energy 
class $\G(X, \omega)\subset \psh(X, \omega)$
 that is a global version of $\G(\Omega)$.
  From the local construction 
 we get global Monge-Amp\`ere currents
$[dd^c \varphi + \omega]^p$ for $\varphi\in \G(X,\omega)$ 
that only depend on the current $dd^c \varphi+ \omega$. 
The limits of  Monge-Amp\`ere currents of certain natural
regularizations of $\varphi$ can be expressed in terms of $[dd^c \varphi + \omega]^j$ for 
$j\leq p$.
We get a mass formula involving the currents $[dd^c
\varphi+\omega]^p$ that describes the loss of mass of the
non-pluripolar Monge-Amp\`ere measure $\la dd^c
\varphi+\omega\ra^n$. 
The class $\G(X, \omega)$ includes $\omega$-psh functions with
analytic singularities and the class $\E(X, \omega)$ of
$\omega$-psh functions of finite energy and certain other convex energy
classes, although it is not convex itself.

\end{abstract}

\maketitle

\section{Introduction}

Let $\Omega$ be a domain in $\C^n$, and let $u$ be a plurisubharmonic
(psh) function on $\Omega$, $u\in \psh(\Omega)$. If $u$ is $C^2$ then $dd^cu$ is a positive form, and the associated
Monge-Amp\`ere measure is defined as the top wedge power of this form
with itself. This positive measure plays a fundamental role in
pluripotential theory  akin to the role played by the Laplacian in ordinary potential theory.
If $u$ is not $C^2$, then $dd^c u$ is no longer a form but a
current. As is well known the wedge product of currents is typically
not well-defined, which raises the question whether it is still possible
to define a Monge-Amp\`ere measure for more general psh functions.

Bedford-Taylor \cite{BT1, BT2} solved this problem when $u$ is
(locally) bounded. Their idea was to define the Monge-Amp\`ere measure
$(dd^c u)^n$ 
recursively. Assume that $T$ is a closed positive current of bidegree
$(j, j)$. Then $T$ has measure coefficients and since $u$
is bounded, $uT$ is a well-defined current and thus so is 
$dd^c(uT)$. Bedford-Taylor proved that
this current is closed and positive. They could then 
recursively define closed positive currents 
$$(dd^c
u)^{p}:=dd^c (u(dd^c u)^{p-1}).
$$
The Monge-Amp\`ere operators $u\mapsto (dd^c u)^p$ have some essential
continuity properties. Bedford-Taylor proved that if $u_\ell$ is any
sequence of psh functions decreasing to $u$, then $(dd^c u_\ell)^p$
converges weakly to $(dd^c u)^p$.

We are interested in the situation when $u$ is not locally
bounded. Demailly \cite{Dem87, Dem93} showed that it is possible to extend the
Bedford-Taylor Monge-Amp\`ere operators to psh functions that are
bounded outside
 ``small''
sets. Moreover, B\l{}ocki, \cite{B06},  and
Cegrell, \cite{C}, characterized the largest class
$\D(\Omega)$
of psh functions on which there is a 
Monge-Amp\`ere operator $u\mapsto (dd^cu)^n$ that is continuous under decreasing
sequences. For instance, functions in $\psh(\Omega)$ that are bounded
outside a compact set in $\Omega$ are in $\D(\Omega)$, see, e.g.,
\cite{B06}. On the other hand psh functions with analytic
singularities, i.e., locally of the form $u=c\log |f|^2 +b$, where
$c>0$, $f$ is a tuple of holomorphic functions, and $b$ is locally
bounded, are not in $\D(\Omega)$ unless their unbounded locus is
discrete, see \cite{B}  or Proposition~\ref{omtenta}.

To handle more singular psh functions Bedford-Taylor \cite{BT3} introduced the notion of
non-pluripolar Monge-Amp\`ere currents. The idea is to capture the 
Monge-Amp\`ere currents  of the ``bounded part'' of $u\in \psh(\Omega)$. Note that for any
$\ell$, $\max(u,-\ell)$ is psh and locally bounded, and thus $(dd^c\max(u,-\ell))^p$ is
well-defined for any $p$.
For each $p\leq n$,  
\begin{equation}\label{husar}
  \la dd^c u \ra ^p:=\lim_{\ell\to
    \infty}1_{\{u>-\ell\}} \big (dd^c \max(u,-\ell)\big )^p 
\end{equation}
is a form with measure coefficients. The existence of the limit follows from the fact that the Monge-Amp\`ere
operators on bounded psh functions are local in the plurifine topology,
i.e., if $u=v$ on a plurifine open set, then $(dd^c u)^p=(dd^c v)^p$ on that
set. 
One serious issue is that the measure coefficients of $\la dd^c u \ra
^p$ 
might be 
not locally finite, as an example due to Kiselman, \cite{K}, 
shows.
If they are locally finite, however, by 
\cite{BEGZ}, $\la dd^c
u\ra^p$ is a closed positive $(p,p)$-current.
For instance, this is the case when $u$ has analytic singularities. 
We refer to the currents $\la dd^c u\ra^p$ as  
\emph{non-pluripolar Monge-Amp\`ere currents}. 

As the name suggests, 
the non-pluripolar Monge-Amp\`ere
currents do not
charge pluripolar sets.
Thus, since $\la dd^c u\ra^p$ cannot capture the behaviour on the singular set
of $u$, they do not 
coincide with Demailly's extensions of $(dd^c u)^p$ in general. 
In particular, it follows 
that the non-pluripolar Monge-Amp\`ere operators $u\mapsto \la dd^c u\ra^p$ 
 are far from being continuous under decreasing sequences in general. 
For instance, if $u=\log|f|^2$, where $f$ is a holomorphic
function, then $\la dd^c u\ra=0$, whereas for any sequence $u_\ell$
decreaing to $u$, $dd^c u_\ell$ converges weakly to $dd^cu$, which by the
Poincar\'e-Lelong formula is the current of integration $[f=0]$
along the divisor of $f$. 

\smallskip 

The purpose of this paper is to introduce a new class of psh functions
together with an extension of the Demailly-Bedford-Taylor
Monge-Amp\`ere operators that capture the singular behaviour.

\begin{df}\label{grisen}
  Let $\Omega$ be a domain in $\C^n$. 
  We say $u\in\psh(\Omega)$ has \emph{locally finite non-pluripolar
    energy}, $u\in \G(\Omega)$ if, for each $j\leq n-1$, $\la
  dd^c u\ra^j$ is locally finite and $u$ is locally integrable with
  respect to $\la
  dd^c u\ra^j$. 
\end{df}

If $u\in \G(\Omega)$ and $j\leq n-1$, then $u\la dd^c u\ra ^j$ is a
well-defined current and thus by mimicking the original constrution by 
Bedford-Taylor we can define generalized Monge-Amp\`ere currents. 

\begin{df}\label{grisprodukt}
  Given $u\in \G(\Omega)$ for $p=1,\ldots, n$ we define
  \begin{equation}\label{pig}
    [dd^c u]^{p}=dd^c(u\la dd^c u\ra ^{p-1})
    \end{equation}
      and
         \begin{equation*}
        \S_{p}(u)=[dd^c u]^{p}-\la dd^c u\ra^{p}.
      \end{equation*}
\end{df}

Using the locality of the non-pluripolar Monge-Amp\`ere operators and the
integrability it follows that $[dd^c u]^{p}$ and $\S_{p}(u)$ are
closed positive currents. 
In particular, $[dd^cu]^p$ dominates 
$\la dd^c u\ra^p$.

Definitions ~\ref{grisen} and ~\ref{grisprodukt} are inspired by the
construction of Monge-Amp\`ere currents in \cite{A,
  AW}. From \cite{AW}*{Proposition~4.1} it follows that psh functions
with analytic singularities have locally finite non-pluripolar energy.
Thus there are functions in $\G(\Omega)$ that
are not in $\D(\Omega)$. 
If $u\in \psh(\Omega)$ has analytic singularities, then the
currents $[dd^cu]^p$ coincide with the
Monge-Amp\`ere currents $(dd^c u)^p$ introduced in \cite{A, AW}. In this case
$\S_p(u)=\1_Z[dd^cu]^p$, where $Z$ is the unbounded locus of $u$.
In \cite{ASWY, AESWY} these Monge-Amp\`ere currents are 
used to understand non-proper intersection theory in terms of
currents. In particular, the Lelong numbers of the currents $[dd^c \log |f|^2]^p$
are certain local intersection numbers, so-called Segre numbers,
associated with the ideal generated by $f$.

In Section ~\ref{lokala} we provide other examples of functions in
$\G(\Omega)$ and also psh functions that are not in $\G(\Omega)$.
For instance, Example ~\ref{utvidgat} shows that there are psh functions
$u\leq v$ such that $u\in \G(\Omega)$ but $v\notin \G(\Omega)$.

\smallskip

Given $u\in \G(\Omega)$, note that $\max(u,-\ell)$ is a
natural sequence of locally bounded psh functions decreasing to $u$,
cf.\ \eqref{husar}. 
Our
first main theorem states that our new Monge-Amp\`ere currents $[dd^c
u]^p$ are the limits of the Monge-Amp\`ere currents of this
regularization.
\begin{thm}\label{regularisera}
 Assume that $u\in \G(\Omega)$. Then
  \begin{equation*}
    \big(dd^c \max(u,-\ell)\big)^p \to [dd^c u]^p, \quad \ell\to\infty.  
  \end{equation*}

  More generally, let $\chi_\ell:\R\to \R$ be a sequence of nondecreasing
convex functions, bounded from below, 
that decreases to $t$
as $\ell\to \infty$, and let $u_\ell=\chi_\ell\circ u$. Then 
\begin{equation*}
  (dd^cu_\ell)^p\to [dd^cu]^p, \quad \ell\to\infty.  
  \end{equation*}
\end{thm}
Note that $u_\ell=\max(u, -\ell)$ corresponds to $\chi_\ell(t)=\max(t, -\ell)$.
Also note that Theorem ~\ref{regularisera} implies that $[dd^c u]^p$
coincides with the Bedford-Taylor-Demailly Monge-Amp\`ere current when this is
defined. 

\begin{ex}
  Let $u=\log|f|^2$, where $f$ is a tuple of holomorphic functions and
  let $\chi_\epsilon=\log(e^t+\epsilon)$.
  Then $\chi_\epsilon \circ u = \log (|f|^2 +\epsilon)$ and Theorem
  ~\ref{regularisera} asserts that
  \[
\lim_{\epsilon\to 0} \big (    dd^c \log (|f|^2 +\epsilon) \big )^p =
[dd^c u ]^p.
\]
This was proved in \cite{A}*{Proposition~4.4}.
\end{ex}
The Monge-Amp\`ere currents of the natural regularizations $\max(u,
-\ell)$ do not always converge, see Example ~\ref{varannan}, and thus not all psh
functions are in $\G(\Omega)$.

Since there are functions in $\G(\Omega)$ that are not in $\D(\Omega)$
 we cannot expect continuity for all decreasing sequences.
Our next result is a twisted version of Theorem  ~\ref{regularisera} that
illustrates that failure of continuity.
Let $v$ be a locally bounded psh function on $\Omega$.
Then $\max(u,v-\ell)$
is another natural sequence of locally bounded psh functions decreasing to $u$. 
Moreover, if $\chi_\ell$ is as in Theorem \ref{regularisera}, then 
also $\chi_\ell\circ (u-v)+v$ is a sequence of locally bounded psh functions decreasing to $u$. 
\begin{thm}\label{vegetera}
Assume that $u\in \G(\Omega)$ and that $v$ is a smooth psh function on $\Omega$. 
Then 
  \begin{equation*}
    \big(dd^c \max(u,v-\ell)\big)^p \to [dd^c u]^p + \sum_{j=1}^{p-1}
  \S_j(u)\w  (dd^c v)^{p-j}, \quad \ell\to\infty.  
  \end{equation*}

More generally, let $\chi_\ell:\R\to \R$ be a sequence of nondecreasing
convex functions, bounded from below, 
that decreases to $t$
as $\ell\to \infty$, and let $u_\ell=\chi_\ell\circ (u-v)+v$.  Then 
\begin{equation*}
   (dd^cu_\ell)^p
  \to [dd^c u]^p + \sum_{j=1}^{p-1}
  \S_j(u) \w (dd^c v)^{p-j}, \quad \ell\to\infty.  
    \end{equation*}
  \end{thm}
  
Note that the lower degree Monge-Amp\`ere currents $[dd^c u]^j$ come
into play. Also note that Theorem ~\ref{regularisera} follows from
Theorem ~\ref{vegetera} by
setting $v=0$.

When $u$ has analytic singularities, Theorem ~\ref{regularisera} first
appeared in \cite{ABW}*{Theorem~1.1}, and Theorem ~\ref{vegetera}
appeared in \cite{B}*{Theorem~1}, although formulated slightly
differently, cf.\ Remark ~\ref{oversattning} below. 
In those papers, using a Hironaka desingularization, the results are reduced to
the case with divisorial singularities.  Such a reduction is not available in the general case,
and in this paper we instead rely on properties from \cite{BEGZ} of the non-pluripolar
Monge-Amp\`ere operator. 
In particular, we get new proofs of the results in \cite{ABW} and
\cite{B}.

\bigskip 

Let us now turn to the global setting.
Assume that $(X,\omega)$ is a compact K\"ahler manifold of dimension $n$. 
Recall that a function $\varphi$ is said to be $\omega$-psh,
$\varphi\in \psh(X,\omega)$, if whenever $h$ is a local potential for
$\omega$, i.e. $dd^c h=\omega$, $\varphi+h$ is psh. Then $dd^c \varphi+\omega$ is a closed positive current in $[\omega]$, and by the $dd^c$-lemma any closed positive current in $[\omega]$ can be written as $dd^c \varphi+\omega$ for some $\omega$-psh $\varphi$, and this $\varphi$ is unique up to adding of constants. Thus studying $\omega$-psh functions is the same as studying closed positive currents in $[\omega]$.

If $\varphi\in \psh(X, \omega)$ is bounded, then there are well-defined Monge-Amp\`ere currents
$(dd^c \varphi+ \omega)^p$, locally defined as $(dd^c
(\varphi+h))^p$,  
where $h$ is a local potential for $\omega$.
It turns out, \cite{BEGZ}*{Proposition~1.6}, that for un unbounded
$\varphi$ the non-pluripolar Monge-Amp\`ere currents $\la dd^c
\varphi+\omega \ra^p$ are always well-defined. 
Moreover, 
\cite{BEGZ}*{Proposition~1.20} showed that
\begin{equation*}
  \int_X \la dd^c \varphi+\omega \ra^p\w\omega^{n-p}\leq \int_X \omega^n.
\end{equation*}
When $\varphi$ is bounded we have equality
\begin{equation}\label{masslikhet}
  \int_X (dd^c \varphi+\omega )^p\w\omega^{n-p} =\int_X \omega^n
\end{equation}
but in general the inequality can be strict. 

Our definitions of $\G(X)$ and Monge-Amp\`ere currents
naturally lend themselves to the global setting.
\begin{df}\label{drottning}
  Let $(X, \omega)$ be a compact K\"ahler manifold of dimension
  $n$. We say that $\varphi\in \psh(X,\omega)$ has \emph{finite
    non-pluripolar energy}, $\varphi\in \G(X,\omega)$, if, for each
  $j\leq n-1$, $\varphi$ is integrable with respect to $\la dd^c \varphi +
  \omega\ra^j$. 
\end{df}
\begin{df}\label{prins}
  Given $\varphi\in \G(X, \omega)$ we define
  \begin{equation*}
    [dd^c \varphi+\omega]^{p}=\big [dd^c(\varphi+h)\big]^p, 
  \end{equation*}
  where $h$ is a local potential for $\omega$, 
  and
  \begin{equation*}
    \S_{p}^\omega(\varphi)=[dd^c \varphi+\omega]^{p}-\la dd^c
    \varphi+\omega\ra ^{p}.
   \end{equation*}
\end{df}
Since two local potentials differ by a pluriharmonic function, $[dd^c \varphi+\omega]^{p}$ and $\S_{p}^\omega(\varphi)$
are well-defined global positive closed currents on $X$. 
Note that whether an $\omega$-psh function $\varphi$ is in $\G(X,
\omega)$ only depends on
the current $dd^c \varphi+\omega$ and not on the choice of $\omega$ as a
K\"ahler representative in the
class $[\omega]$. Also the currents $[dd^c \varphi+\omega]^{p}$ and
$\S_{p}^\omega(\varphi)$ only depend on
the current $dd^c \varphi+\omega$.

\smallskip 

From Theorem ~\ref{vegetera} we get global 
regularization results.
Given $\varphi\in \psh(X, \omega)$, note that $\max (\varphi, -\ell)$
is a natural sequence of bounded $\omega$-psh
functions decreasing to $\varphi$.

\begin{thm}\label{krubba}
  Assume that $\varphi\in \G(X, \omega)$. Then
  \begin{equation*}
 \big (dd^c \max (\varphi, -\ell) + \omega \big )^p \to  
[dd^c \varphi+\omega]^p+\sum_{j=1}^{p-1} \S_j^\omega
(\varphi)\w \omega^{p-j}, \quad \ell\to\infty.   
\end{equation*}

More generally, let $\chi_\ell:\R\to \R$ be a sequence of nondecreasing
convex functions, bounded from below, 
that decreases to $t$
as $\ell\to \infty$, and let $\varphi_\ell=\chi_\ell\circ \varphi$. Then 
\begin{equation*}\label{kokos}
  (dd^c\varphi_\ell+\omega)^p\to
[dd^c \varphi+\omega]^p+\sum_{j=1}^{p-1} \S_j^\omega
(\varphi)\w\omega^{p-j}, \quad \ell\to\infty.  
\end{equation*}
\end{thm}

If $\eta$ is another K\"ahler form in $[\omega]$, then
$\eta=\omega+dd^c g$ for some smooth function $g$. There is an
associated regularization of $\varphi$, namely
$\varphi_\ell:=\max(\varphi-g,-\ell)+g$, which corresponds to the max-regularization of the current $dd^c \varphi+\omega$ with
respect to the alternative decomposition $dd^c(\varphi-g)+\eta$.

\begin{thm}\label{kyrka}
 Assume that $\varphi\in \G(X, \omega)$, that $\eta$ is a K\"ahler
 form in $[\omega]$, and that $g$ and
 $\varphi_\ell$ are as above. Then
  \begin{equation}\label{saffran}
 (dd^c \varphi_\ell +\omega)^p \to  
[dd^c \varphi+\omega]^p+\sum_{j=1}^{p-1} \S_j^\omega
(\varphi)\w \eta^{p-j}, \quad \ell\to\infty.   
\end{equation}

More generally, let $\chi_\ell:\R\to \R$ be a sequence of nondecreasing
convex functions, bounded from below, 
that decreases to $t$
as $\ell\to \infty$, and let $\varphi_\ell=\chi_\ell\circ
(\varphi-g)+g$. Then \eqref{saffran} holds. 
\end{thm}
Note that Theorem ~\ref{krubba} follows immediately from Theorem
~\ref{kyrka} by setting $g=0$.
As in the local case, 
for $\varphi$ with analytic singularities Theorems ~\ref{krubba} and
~\ref{kyrka}  
follow from
\cite{B}*{Theorem~1}, cf.\ Remark 
~\ref{overstyv}.

From \eqref{masslikhet} and Theorem ~\ref{krubba} 
we get the following mass formula.

\begin{thm}\label{mannagryn}
Assume that $\varphi\in
\G(X, \omega)$. Then for each $p\leq n$, 
\begin{equation*}
\int_X \la dd^c \varphi+\omega\ra^p\w \omega^{n-p}+\sum_{j=1}^p \int_X
\S_j^\omega(\varphi)\w \omega^{n-j} =\int_X\omega^n.
\end{equation*}
\end{thm}
In fact, Theorem ~\ref{mannagryn} is a cohomological consequence of
the definition of $[dd^c \varphi+\omega]^p$; in Section ~\ref{vattenflaska} we provide a direct proof that does not rely on 
Theorem ~\ref{krubba}. 
For $\varphi$ with analytic singularities this theorem appeared 
in \cite{ABW}*{Theorem~1.2 and Proposition~5.2}.
Note that, for $j\leq p$, the current $S_j^\omega(\varphi)$ captures
the mass that ``escapes'' from $\la dd^c \varphi+\omega\ra ^p$ at
codimension $j$. 

\bigskip

From the local case it follows that $\omega$-psh functions with analytic
singularities are in $\G(X, \omega)$, and in Section
\ref{exempelsamling} we provide other examples.
However, in the global
setting we know more about the structure of the class $\G(X, \omega)$.
In particular, it contains the B\l{}ocki-Cegrell class. 
Note that being in the B\l{}ocki-Cegrell class is a
local statement, cf.\ Proposition ~\ref{humle} below. We say that $\varphi\in \psh(X, \omega)$ is in $\D(X, \omega)$ if
    whenever $g$ is a local $dd^c$-potential of $\omega$ in an open
    set $\U\subset X$, then $\varphi+g\in \D(\U)$. 
\begin{thm}\label{globalbc}
  Let $(X, \omega)$ be a compact K\"ahler manifold of dimension $n$.
  Then $\D(X, \omega)\subset \G(X, \omega)$.
  \end{thm}

Next, as the name suggests, the class $\G(X, \omega)$ of $\omega$-psh
functions with finite non-pluripolar energy 
can be understood as a
finite energy class.
Recall that the \emph{Monge-Amp\`ere energy} of $\varphi\in
\psh(X,\omega)$,  
introduced in \cite{GZ}, inspired by earlier work \cite{cegrell} in
the local setting, is defined as 
\begin{equation}\label{skramla}
  E(\varphi)=\frac{1}{n+1}\sum_{j=0}^n\int_X
  \varphi(dd^c\varphi+\omega)^j\wedge \omega^{n-j} 
\end{equation}
if $\varphi$ is bounded and by 
  \begin{equation}\label{ramla}
    E(\varphi)=\inf\left \{E(\psi): \psi\geq \varphi, \psi\in \psh(X,\omega)\cap
    L^{\infty}(X)\right \}
\end{equation}
in general. 
The corresponding \emph{finite energy class} 
  \begin{equation}\label{famla}
    \E(X,\omega):=\{\varphi\in \psh(X,\omega):
    E(\varphi)>-\infty\}
  \end{equation}
  is convex.
Recall that, if $\varphi, \psi\in \psh(X, \omega)$, then $\varphi$ is
said to be \emph{less singular} than $\psi$, $\varphi\succeq \psi$, if $\varphi\geq \psi +
O(1)$. If $\varphi\succeq \psi$ and $\psi\succeq\varphi$ we say that $\varphi$ and $\psi$ have the same
\emph{singularity type} and write $\varphi\sim \psi$.
The class $\E(X, \omega)$ is  
  closed under finite pertubations in the sense that if
  $\varphi\in \E(X, \omega)$ and $\psi\sim \varphi$,
   then
  $\psi\in\E(X, \omega)$. 
Moreover, $\E(X, \omega)$ 
is contained in the \emph{full mass class} 
  \begin{equation}\label{svamla}
    \F(X,\omega):=\left \{\varphi\in \psh(X,\omega): \int_X
    \langle dd^c \varphi+\omega\rangle^ n=\int_X
    \omega^n\right \}.
\end{equation}

We introduce an alternative energy for $\varphi\in \psh(X, \omega)$. 
\begin{df}\label{mamba} 
  Let $(X, \omega)$ be a compact K\"ahler manifold of dimension
  $n$. For $\varphi\in \psh(X, \omega)$ we define the \emph{non-pluripolar energy} 
\begin{equation*}
E^{np}(\varphi)=\frac{1}{n}\sum_{j=0}^{n-1}\int_X \varphi\langle dd^c
\varphi+\omega\rangle^j\wedge \omega^{n-j}.
\end{equation*}
\end{df} 
Note that 
\begin{equation*}
  \G(X,\omega)=\{\varphi\in \psh(X,\omega):
  E^{np}(\varphi)>-\infty\}, 
\end{equation*}
so that $\G(X, \omega)$ can be thought of as an finite energy class,
cf. \eqref{famla}.
\begin{thm}\label{stranda}
  Let $(X, \omega)$ be a compact K\"ahler manifold. Then 
  \begin{enumerate}
  \item
    if $\varphi\in
    \G(X,\omega)$ and $\psi\sim \varphi$, then $\psi\in
    \G (X,\omega)$; 
  \item
    $\E(X,\omega)\subset \G(X,\omega)$. 
    \end{enumerate} 
  \end{thm}

Although $\G(X, \omega)$ contains the convex subclass $\E(X, \omega)$
it is not convex itself.
However, it contains certain other convex energy classes. 
%
  \begin{df}\label{flamenco} 
  Let $\psi\in\G(X, \omega)$.
For $\varphi\in \psh(X,\omega)$, such that $\varphi\preceq \psi$, we
define the \emph{energy relative to $\psi$} 
  \begin{equation*}
    E^{\psi}(\varphi)=\inf\{E^{np}(\varphi'): \varphi'\geq 
    \varphi, \varphi'\sim \psi\}.\end{equation*}
We define the corresponding \emph{finite relative energy classes} 
  \begin{equation*}
  \E^\psi(X, \omega) =\{\varphi\preceq \psi,
  E^\psi(\varphi)>-\infty\}
  \end{equation*}
  and the \emph{relative full mass classes} 
    \begin{multline*}
      \F^{\psi}(X,\omega)=\\\Big \{\varphi\in \psh(X,\omega): \varphi\preceq
      \psi, \sum_{j=0}^{n-1}\int_X \langle dd^c \varphi+\omega\rangle^j \wedge
      \omega^{n-j}=\sum_{j=0}^{n-1}\int_X \langle dd^c \psi+\omega\rangle^j\wedge
      \omega^{n-j} \Big \}.
      \end{multline*}
\end{df}
  Note that if $\psi\in \psh (X,
  \omega)\cap L^\infty (X)$, then $\E^\psi(X, \omega)\supset \E(X,
  \omega)$. 
The classes $\E^\psi(X, \omega)$ have the following properties,
similar to $\E(X, \omega)$. Following \cite{BEGZ} we say that $\varphi\in \psh(X, \omega)$ 
has \emph{small unbounded locus (sul)} if there exists a
complete pluripolar closed subset $A\subset X$ such that
$\varphi$ is locally bounded outside $A$, cf.\ Section
\ref{struntsumma} below. 
\begin{thm}\label{blanda}
  Let $(X, \omega)$ be a compact K\"ahler manifold. Then 
      \begin{enumerate}
      \item
        if $\varphi\in \E^\psi(X,\omega)$ and $\varphi' \sim \varphi$, 
        then $\varphi'\in \E^\psi(X,\omega)$;
      \item
        if $\psi$ has sul, 
then $\E^\psi(X,\omega)$ is convex;
      \item
        $\E^\psi(X,\omega)= \G (X,\omega)\cap \F^\psi(X,\omega)$.
  \end{enumerate}
      \end{thm}

\smallskip

The paper is organized as follows. In Section ~\ref{nonp} we provide
some background on the classical and the non-pluripolar Monge-Amp\`ere operators in the local
setting. In Section ~\ref{zooma} we introduce the class $\G(\Omega)$, 
and more generally classes $\G_k(\Omega)$ of psh functions of \emph{locally finite
 non-pluripolar energy of order $k$}, 
and our Monge-Amp\`ere operators, and in Section ~\ref{lokala} we
provide various examples of functions in $\G(\Omega)$. In Section
~\ref{regsec} we prove the regularity result Theorem
~\ref{vegetera}. In fact, we prove a sligthly more general version
formulated in terms of $\G_k(\Omega)$. 

In Section ~\ref{globala} we extend our definitions and regularity results to the
global setting. In Section ~\ref{maenergi} we recall the classical
 Monge-Amp\`ere energy and in Sections ~\ref{ickepp} and ~\ref{relativa} we study the non-pluripolar
energy and the relative energy, respectively; in particular, we prove
Theorems ~\ref{stranda} and ~\ref{blanda}. As in the local case we
introduce more generally non-pluripolar and relative energies of order
$k$ and corresponding finite energy classes $\G_k(X, \omega)$ and
$\E_k^\psi(X, \omega)$, and we prove versions of our results formulated in
terms of these. 
In Section ~\ref{dumle} we discuss
the B\l{}ocki-Cegrell class and prove Theorem
~\ref{globalbc}. Finally, 
in Section ~\ref{exempelsamling} we give various examples of functions with finite
non-pluripolar energy.

\smallskip
\noindent {\bf Ackowledgement}
We would like to thank the referees for valuable comments and suggestions
that have improved the presentation.

\section{The complex Monge-Amp\`ere product}\label{nonp}
Throughout this paper $X$ is a domain in $\C^n$ or
more generally a complex manifold of
dimension $n$. 
All measures are assumed
to be Borel measures. We let $d^c={1/4\pi i}(\partial - \dbar)$, so
that $dd^c \log |z_1|^2=[z_1=0]$.

In this section we recall some basic facts about the (non-pluripolar)
Monge-Amp\`ere products. We refer to, e.g., \cite{D}*{Chapter III} for the classical
Bedford-Taylor-Demailly theory, see also \cite{BT2, BT3, BEGZ}. 
%

First, the plurifine topology is the coarsest topology such that all psh functions on all open subsets of $X$ are 
continuous. A basis for this topology is given by all sets of the form
$V\cap \{u>0\}$, where $V$ is open and $u$ is psh in
$V$.

Let $T$ be a closed positive current and let $u$ be a locally bounded
psh function on $X$. Then $uT$ is a well-defined current and 
\begin{equation*}
  dd^c u\w T:= dd^c (uT)
\end{equation*}
is again a closed positive current.
In particular, if $u_1, \ldots, u_p$ are locally bounded
psh functions, then the product  $dd^c u_p\w \cdots \w dd^c u_1\w T$
is defined inductively as
\begin{equation}\label{mixad}
dd^c u_{p} \w \cdots \w dd^c u_1\w T:= dd^c( u_{p} dd^c u_{p-1}\w
\cdots \w dd^c u_1\w T).
\end{equation}
It turns out that this product is commutative in the factors $dd^c
u_j$ and
multilinear in the factors $dd^c u_j$ and it does not charge pluripolar sets.
Moreover  $dd^c u_p\w \cdots \w dd^c u_1$ is \emph{local in the plurifine topology}, i.e.,
if $O$ is a plurifine open set and $u_j=v_j$ pointwise on $O$, then 
$$
\1_O dd^c u_p\w \cdots \w dd^c u_1=\1_Odd^c v_p\w \cdots \w dd^c
v_1.
$$
The products \eqref{mixad} satisfy the following continuity property.
\begin{lma}\label{konto}
Assume that $u_1,\ldots, u_p$ are locally bounded psh functions and
that $u_1^\ell, \ldots, u_p^\ell$ are decreasing sequences of psh
functions converging to $u_1,\ldots, u_p$, respectively, and that $T$ is a closed positive
current. Then 
\[
  dd^c u^\ell_p\w \ldots \w dd^c u^\ell_1\w T \to dd^c u_p\w \ldots \w dd^c
  u_1\w T
\]
weakly when $\ell\to \infty$. 
  \end{lma} 

We will use the following result. 
  
  \begin{lma}\label{pronto}
Assume that for $j=1,\ldots, p$, $u_s^j$, $s\in \R$ (or some interval
in $\R$), is a family of locally bounded psh functions on $X$ 
such that $u_s^j\to u^j_t$ locally uniformly on $X$ when $s\to
t$. Then
\[dd^c u_{s_p}^p\w\cdots\w dd^c u_{s_1}^1 \to 
  dd^c u_{t_p}^p\w\cdots\w dd^c u_{t_1}^1
\]
weakly 
when $(s_1,\ldots, s_p)\to (t_1,\ldots, t_p)$ in $\R^p$. 
\end{lma} 
\begin{proof}
  Note that
  \begin{multline}\label{koppar}
    dd^c u_{s_p}^p\w\cdots\w dd^c u_{s_1}^1 -
  dd^c u_{t_p}^p\w\cdots\w dd^c u_{t_1}^1
  =\\
  \sum_{j=1}^p
  \big (dd^c u_{s_p}^p\w\cdots\w dd^c u_{s_j}^j\w
  dd^c u_{t_{j-1}}^{j-1}\w\cdots\w dd^c u_{t_1}^1
  -\\
 dd^c u_{s_p}^p\w\cdots\w dd^c u_{s_{j+1}}^{j+1}\w
  dd^c u_{t_{j}}^{j}\w\cdots\w dd^c u_{t_1}^1 \big ).
  \end{multline}
 Let $\xi$ be a test form. Since \eqref{mixad} is commutative in the
 factors $dd^c u_j$ we can
 write the action of the $j$th term in the right hand side of \eqref{koppar} on $\xi$ as
 \begin{equation}\label{pippi}
 \int_X (u^j_{s_j}-u^j_{t_j}) dd^c u_{s_p}^p\w\cdots\w dd^c u_{s_{j+1}}^{j+1}\w
  dd^c u_{t_{j-1}}^{j-1}\w\cdots\w dd^c u_{t_1}^1 \w dd^c\xi.
\end{equation}
Let $\mathcal U\subset X$ be a relatively compact neighborhood of the support
of $\xi$. Then there is a neighborhood $\mathcal V\subset \R$ of $t$
such that for $j=1,\ldots, p$, $s\in \mathcal V$,
$|u_s^j-v_s^j|<\epsilon$ in $\overline{\mathcal U}$. Since $u_s^j$ are
locally bounded, there is an $M$ such that
$\| u_s^j\|_{L^\infty(\overline{\mathcal U})}\leq M$.  We may also
  assume that $\int_X |dd^c \xi|\leq M$. 
Now by the Chern-Levine-Nirenberg inequalities there is constant $C$
such that the absolute value of \eqref{pippi} is bounded by
\begin{multline*}
  C \sup_{\overline{\mathcal U}}|u_{s_j}^j-u_{t_j}^j| \| u_{s_p}^p\|_{L^\infty(\overline{\mathcal U})}
  \cdots
  \| u_{s_{j+1}}^{j+1}\|_{L^\infty(\overline{\mathcal U})}
  \| u_{t_{j-1}}^{j-1}\|_{L^\infty(\overline{\mathcal U})}
  \cdots
  \| u_{t_1}^{1}\|_{L^\infty(\overline{\mathcal U})}
  \int_X |dd^c \xi|\leq\\
 C \epsilon (M+\epsilon)^{p-j} M^j\to 0.
\end{multline*}
Since this holds for any $\xi$, \eqref{koppar} converges weakly to $0$
when $(s_1,\ldots, s_p)\to (t_1, \ldots, t_p)$. 
  \end{proof}

\subsection{The non-pluripolar Monge-Amp\`ere product}\label{pruttkudde}
Let $u_1,\ldots, u_p$ be not necessarily locally bounded psh functions on $X$
and let
\begin{equation}\label{uttomning}
O_\ell=\bigcap_{j=1}^p \{u_j > -\ell\}.
\end{equation}
Then $O_\ell$ is a plurifine open set.
Following   \cite[Definition~1.1]{BEGZ}
we say that the non-pluripolar Monge-Amp\`ere product $\la dd^c u_p\wedge
\cdots \wedge dd^c u_1\ra$ is well-defined if
for each compact subset $K\subset X$ we have 
\begin{equation*}
\sup_\ell \int_{K\cap O_\ell} \omega^{n-p}\wedge \bigwedge_{j=1}^p dd^c
\max (u_j,-\ell)<\infty, 
\end{equation*} 
where $\omega$ is a smooth strictly positive $(1,1)$-form on $X$.
This definition is clearly independent of $\omega$.

Since the \ma product for bounded functions is local in the plurifine
topology, it follows that 
\begin{equation*}
  \1_{O_\ell}
  \bigwedge_{j=1}^p dd^c \max (u_j,-\ell )
  =\1_{O_\ell} \bigwedge_{j=1}^p dd^c \max (u_j,-\ell' ), \quad \ell'>\ell.
\end{equation*}
It follows that 
there is a well-defined positive $(p,p)$-current 
\begin{equation}\label{husqvarna}
\langle dd^c u_p\w\cdots\w dd^cu_1 \rangle:= \lim_\ell  \1_{O_\ell}
\bigwedge_{j=1}^p dd^c \max (u_j,-\ell ); 
\end{equation} 
by \cite{BEGZ}*{Theorem~1.8} it is closed.

Note that \eqref{husqvarna} is commutative in the factors $dd^cu_j$ since
\eqref{mixad} is.
By \cite{BEGZ}*{Proposition~1.4} it is multilinear 
in the following sense: 
  if $v$ is another psh function, then 
  \begin{equation}\label{beundra}
    \Big \la dd^c (u_p+v) \w \bigwedge_{j=1}^{p-1} dd^c u_{j}\Big \ra
    =
    \Big \la dd^c u_{p}\w \bigwedge_{j=1}^{p-1} dd^c u_{j}\Big \ra
    +
    \Big \la dd^c v \w \bigwedge_{j=1}^{p-1} dd^c u_{j}\Big \ra 
  \end{equation}
  in the sense that the left hand side is well-defined if and only if
  both terms in the right hand side are, and equality holds in this
  case. 
Moreover, \eqref{husqvarna} only depends on the currents $dd^c
  u_j$, i.e., it is not affected by adding pluriharmonic functions to
  the $u_j$. Also, the operator 
  \begin{equation*}
    (u_1,\ldots, u_p)\mapsto \langle dd^c u_p\w\cdots\w dd^cu_1
    \rangle
  \end{equation*}
  is local in the plurifine
  topology whenever it is well-defined.

    \begin{lma}\label{handduk}
Assume that $u$ is a psh function on $X$ such that $\la dd^c u\ra^p$ is
well-defined, 
and that $u_\lambda$ is a sequence of psh
functions on $X$ decreasing to $u$, such that $\la dd^c u_\lambda\ra^p$
is well-defined for each $\lambda$.
Moreover, assume that $\omega$ is a smooth positive $(1,1)$-form and
that $\chi$ is a non-negative test function on $X$. 
Then 
\begin{equation}\label{vedklabb}
 \liminf_{\lambda\to \infty} \int_X\la dd^c u_\lambda\ra^p \w\chi \omega^{n-p}\geq
 \int_X \la dd^c u
 \ra^p \w \chi \omega^{n-p}.
\end{equation}
\end{lma}

\begin{proof}
Fix $\ell$ and let $O_\ell=\{u>-\ell\}$. Then $\max (u_\lambda,
-\ell)\searrow \max (u, -\ell)$ and by \cite{BT3}*{Corollary~3.3}, 
\begin{equation}\label{spraka}
   \liminf_{\lambda\to\infty}
   \int_{O_\ell}\big (dd^c \max (u_\lambda, -\ell) \big )^p\w \chi\omega^{n-p}
   \geq 
  \int_{O_\ell}\big (dd^c \max (u, -\ell)  \big )^p\w \chi\omega^{n-p}.  
   \end{equation}
Since $u_\lambda>u$, $u_\lambda = \max (u_\lambda,
-\ell)$ and $u=\max (u, -\ell)$ in $O_\ell$. Thus,
since  the non-pluripolar Monge-Amp\`ere operator is local in the plurifine topology,
\eqref{spraka} implies that 
\begin{equation}\label{chai}
   \liminf_{\lambda\to\infty}
   \int_{O_\ell}\la dd^c u_\lambda\ra^p\w \chi\omega^{n-p}
   \geq 
  \int_{O_\ell}\la dd^c u \ra^p\w \chi\omega^{n-p}.  
   \end{equation} 
   Now, \eqref{chai} holds for all $\ell$ and since $\la dd^c u\ra^p$
   does not charge pluripolar sets, in particular not 
   $V=\{u=-\infty\}$, we get \eqref{vedklabb}. 
  \end{proof}

\subsection{Psh functions with small unbounded locus}\label{struntsumma}
Following \cite{BEGZ} we say that a psh function $u$ on $X$ has \emph{small
  unbounded locus (sul)} if
there exists a complete pluripolar closed subset $A\subset
X$ such that $u$ is locally bounded outside $A$.

\begin{remark}\label{pojkar} 
Let $O=\bigcup_\ell O_\ell$, where $O_\ell$ is defined by
\eqref{uttomning}. Note that if $u_1,\ldots, u_p$ have sul and $A$ is
a closed complete pluripolar set such that $u_1,\ldots, u_p$ are locally
bounded outside $A$, then $X\setminus O\subset A$.
\end{remark} 

\begin{remark}\label{flickor}
  Assume that $u_1,\ldots, u_p$ have sul, and that $A$ is a closed complete
    pluripolar closed set such that each $u_j$ is locally bounded outside
  $A$. Then 
    $\la dd^c u_p\w\cdots \w dd^c u_1 \ra $ is
well-defined if and only if the
Bedford-Taylor product $dd^c u_p\w\cdots \w dd^c u_1 $,  which is defined on $X\setminus A$,
has locally finite mass near each point of $A$. Then $\la dd^c
u_p\w\cdots \w dd^c u_1 \ra $
is just the trivial extension of $dd^c u_p\w\cdots \w dd^c u_1 $, cf.\ 
\cite[page 204]{BEGZ}. 
\end{remark}

\section{Local Monge-Amp\`ere currents - the classes
  $\G_k(X)$}\label{zooma}
We slightly extend the definition of $\G(X)$ from the introduction,
cf.\ Definition ~\ref{grisen}.

\begin{df}\label{prisen}
 Let $X$ be a complex manifold of dimension $n$.
For $1\leq k\leq n-1$, we say that a psh function $u$ on $X$ has
\emph{locally finite non-pluripolar energy of order k}, $u\in\G_k(X)$,
if, for each $ j\leq k$,  $\la dd^c u\ra^j$ is locally finite and
$u$ is locally integrable with respect to $\la dd^c u\ra^j$. 
\end{df}

Note that if $\omega$ is a smooth strictly positive $(1,1)$-form, then $u\in
\G_k(X)$ if and only if 
the measure
\begin{equation}\label{grizzly}
  \sum_{j=0}^{k}\la dd^c  u\ra^j\w \omega^{n-j}
\end{equation}
is locally finite and $u$ is locally
integrable with respect to this measure. 
Clearly
\begin{equation*}
\G_1(X)\supset \G_2(X) \supset \cdots\supset \G_{n-1}(X)=\G(X),
   \end{equation*}
  where $\G(X)$ is as in Definition ~\ref{grisen}. 

\smallskip

If $u\in \G_k(X)$, then the Monge-Amp\`ere currents
\[
  [dd^c u]^p=dd^c( u \la dd^c u\ra^{p-1}) 
  \]
  and
  \[
    S_p(u)= [dd^c u]^p- \la dd^c u\ra^{p}
  \]
  from Definition ~\ref{grisprodukt} are well-defined $(p,p)$-currents
  for $p=1,\ldots, k+1$. 
\begin{prop}\label{slutna}
  The currents $[dd^cu]^p$ and $\S_p(u)$ are closed and positive. 
  \end{prop}
  \begin{proof}
Clearly $[dd^cu]^p$ is closed and, since $\la dd^c u\ra^p$ is closed,
so is $\S_p(u)$.

Note that if $u_\ell$ is a sequence of smooth psh functions decreasing
to $u$, then
\[
  [dd^c u]^p=\lim_{\ell\to \infty} dd^c (u_\ell \la dd^c u\ra^{p-1}).
\]
Since $\la dd^c u\ra^{p-1}$ is closed and positive, the currents in the right hand side are positive and thus so is the
limit. 
To see that also $S_p(u)$ is positive, let 
$u_\ell=\max(u,-\ell)$ and $O_\ell=\{u>-\ell\}$. 
Since $u_\ell\to u$, 
\begin{multline}\label{lastlimit}
[dd^cu]^p=\lim_{\ell\to \infty} dd^c (u_\ell \la dd^c u\ra^{p-1})=\\
\lim_{\ell\to\infty} \1_{X\setminus O_\ell} dd^c (u_\ell \la dd^c u\ra^{p-1})+\lim_{\ell\to\infty} \1_{O_\ell} dd^c(u_\ell \la dd^c u\ra^{p-1}).
\end{multline} 
Since the non-pluripolar Monge-Amp\`ere operator 
is local in the plurifine topology, 
$\1_{O_\ell} dd^c(u_\ell \la dd^c u\ra^{p-1})=\1_{O_\ell}dd^c u_\ell\w  \la dd^c u\ra^{p-1}=
\1_{O_\ell}\la dd^c u_\ell\ra^{p}$, 
and hence the last limit in \eqref{lastlimit} is equal to $\la dd^c u\ra^p$.
Hence the first limit 
must exist as well and it is certainly positive.
\end{proof}

\begin{prop}\label{obero}
  The currents $[dd^cu]^p$ and $\S_p(u)$ only depend on $dd^cu$,
  i.e., they are not affected by adding a pluriharmonic function to
  $u$. 
\end{prop}

\begin{proof}
Since
$\la dd^c u\ra^p$ only depends on
$dd^cu$, cf.\ Section ~\ref{pruttkudde}, it is enough to prove the proposition for $[dd^c u]^p$.
Assume that $u'=u+h$, where $h$ is pluriharmonic.
Then by \eqref{beundra}, 
\begin{equation*}
  [dd^cu']^p=dd^c \big ( (u+h)\la dd^c u\ra^{p-1} \big )=
   dd^c  ( u\la dd^c u\ra^{p-1} ) + 
  dd^c  ( h\la dd^c u\ra^{p-1} ) =
  [dd^cu]^p,
\end{equation*}
where the last equality follows since $dd^c h=0$ and $\la dd^c
u\ra^{p-1}$ is closed.
  \end{proof}

 \begin{remark}
    Assume that $u\in \G(X)$ has sul and is locally bounded outside
    the closed 
complete pluripolar set $A\subset X$.
Note that $[dd^c u]^p$ coincides with the standard Monge-Amp\`ere
current $(dd^c u)^p$ outside $A$. It follows from Remark ~\ref{flickor}
that 
\begin{equation*}
\1_{X\setminus A} [dd^c u]^p= \la dd^c u\ra^p \quad \text{ and }
\quad \S_p (u)=\1_A [dd^c u]^p. 
  \end{equation*}
\end{remark}

\section{Examples}\label{lokala} 
Let us consider some examples of functions with locally finite
non-pluripolar energy.

\begin{ex}\label{aladab}
Assume that $u\in \psh (X)$ has analytic singularities, i.e., that $u$
is locally of
the form 
\begin{equation}\label{himmelsbla}
  u=c\log |f|^2 +b,
\end{equation}
where $c>0$, $f=(f_1, \ldots, f_m)$ is a tuple of holomorphic functions,
$|f|^2=|f_1|^2+\cdots + |f_m|^2$, and $b$ is locally bounded.
Then $u$ is locally bounded outside the variety $Z\subset X$, locally
defined as $\{f=0\}$; in particular, $u$ has sul.

As mentioned in the introduction, it follows from
\cite{AW}*{Proposition~4.1} that $u\in \G(X)$. Moreover the currents $[dd^c u]^p$
coincide with the Monge-Amp\`ere currents $(dd^c u)^p$ 
defined inductively in \cite{AW} as $(dd^c u)^p=dd^c(u\1_{X\setminus
  Z} (dd^c u)^{p-1})$.

For the reader's convenience let us sketch an argument.  Assume
that \eqref{himmelsbla} holds in the open set $\U\subset X$ and, for
simplicity, that $c=1$.
By Hironaka's theorem there is a smooth 
modification $\pi\colon \U'\to \U$ such that $\pi^* f=f^0 f'$, where $f^0$ is a
 holomorphic section of a line bundle $L\to \U'$ and $f'$ is a non-vanishing tuple
 of holomorphic sections of $L^{-1}$.
Now
$$
\pi^* u = \log|f^0|^2 + b', 
$$
where $b'=\log |f'|^2+\pi^* b$ is locally bounded and psh in any local
frame for $L$.  
It follows that, for any $p$, $(dd^c b')^p$ is a well-defined closed positive current on $\U'$ and one can
check that
\[
\pi_* (dd^c b')^p =  \la dd^c u\ra^p.
  \]
Thus to see that $u\in \G(\U)$ it is enough to verify that $\log
|f^0|^2 (dd^c b')^p$ has locally bounded mass and this follows from a
standard Chern-Levine-Nirenberg type estimate, see, e.g., \cite[Chapter
III, Proposition~3.11]{D}. 
By the Poincar\'e-Lelong formula, 
\begin{equation*}
  \S_p (u) = \pi_* \big ( dd^c (\log |f^0|^2 (dd^c b')^p \big )
  =
  \pi_* \big ([\div f^0]\w (dd^c b')^p \big ), 
  \end{equation*} 
  where $[\div f^0]$ is the current of integration along
  the divisor 
  of $f^0$.
\end{ex}

\def\k{{\kappa}}

\begin{prop}\label{omtenta} 
Let $\Omega$ be a domain in $\C^n$. Assume that $u\in \psh(\Omega)$
has analytic singularities and that the unbounded locus of $u$ is not
discrete. Then $u$ is not in 
$\mathcal D(\Omega)$. 
   \end{prop}

This result was first noted in \cite{B}. Here we provide a different argument. 

   \begin{proof}
Let $\k=\codim Z$. We claim that the current
$S_{\k}(u)=\1_Z [dd^c u]^\k=\1_Z(dd^cu)^\k$, where $(dd^cu)^\k$ is the
classical Bedford-Taylor-Demailly Monge-Amp\`ere current, is non-zero. Taking this for granted,
since $\kappa<n$ and all $S_j(u)$ are positive currents, it follows from
Theorem~1.5 that we can find a decreasing sequence $u_\ell$
converging to $u$ such that $(dd^c u_\ell)^n$ does not converge to $[dd^c u]^n$.
We conclude that  $u$ is not in $\mathcal D(\U)$,
cf.\ the introduction.

To prove the claim  let us assume that \eqref{himmelsbla} holds in the
open set $\mathcal U\subset \Omega$. Now the Lelong numbers of
$S_{\k}(u)$ and $\1_{Z}(dd^c\log|f|^2)^\k$ coincide at each point in
$\mathcal U$, see, e.g.,
\cite{AW}*{(1.9)}. For dimension reasons, both currents must be Lelong
currents, and thus $S_{\k}(u)=\1_{Z}(dd^c\log|f|^2)^\k$. By the
classical King formula,
see, e.g., \cite[Chapter III, (8.18)]{D}, $\1_{Z}(dd^c\log|f|^2)^\k$ is the Lelong current of an effective  cycle
whose support is precisely the union of the irreducible components of
$Z$ of pure codimension $\k$. 
In particular $S_\k(u)$ is non-zero. 
\end{proof}

Next, let us consider some examples of functions in $\G(X)$ that do not have analytic singularites.

\begin{ex}\label{utvidgat}
  Let $f$
   be a tuple of holomorphic functions in
  a domain $\Omega\subset \C^n$
such that $|f|^2<1$
and let 
$$
u=-(-\log|f|^2)^\epsilon
$$
for some $\epsilon\in(0,1)$. Then $u$ is psh in $\Omega$ and it is locally bounded
outside $Z=\{f=0\}$; in particular $u$ has sul.
We claim that for each $k$, $u\in \G_k(\Omega)$ if and only if
$\epsilon <1/2$. Moreover, although $v:=\log |f|^2 \in \G(\Omega)$,
cf.\ Example ~\ref{aladab}, 
$u+v$ is not in $\G_k(\Omega)$
for any $\epsilon$.

\smallskip 

To prove the claim, first note that 
\begin{equation*}
dd^c u=
\frac{i}{2\pi}\epsilon(1-\epsilon) (-\log|f|^2)^{\epsilon-2}
\frac{ \partial |f|^2 \w \dbar |f|^2}{|f|^4}
+
\epsilon (-\log|f|^2)^{\epsilon-1}dd^c \log|f|^2.
\end{equation*}
Next, note that if $\pi:X\to \Omega$ is a smooth modification, then
$u\la dd^c u \ra^j$ has locally finite mass 
if and only if
$\pi^*(u\la dd^c u \ra^j)=\pi^*u\1_{X\setminus \pi^{-1}Z}(dd^c \pi^*
u)^j$ has locally finite mass. 
By Hironaka's theorem there is such a
modification so that 
$\pi^*f=f^0f'$, where $f^0$ is a holomorphic section of a line bundle
$L$ and $f'$ is a non-vanishing tuple of holomorphic sections of $L^{-1}$.
Given a local frame we may assume that 
$f^0$ and $f'$ is a function and a tuple of functions, respectively.
Let $\eta=\dbar |f'|^2/|f'|^2$.
Then
\begin{equation*}
  \pi^* \Big (\frac{ \partial |f|^2 \w \dbar |f|^2}{|f|^4} \Big )
  =
  \frac{ d f^0 \w d \overline {f^0}}{|f^0|^2}
  +
  \frac{ d f^0 }{f^0}\w \eta
  +
  \overline{ \frac{ d f^0 }{f^0}\w \eta}
  +
  \bar\eta \w \eta  
\end{equation*}
   and, by the Poincar\'e-Lelong formula,  
  \begin{equation*}
    \pi^* (dd^c \log |f|^2)= 
    [D]+\omega_f,
  \end{equation*}
  where $D$ is the divisor defined by $f^0$ and $\omega_f:=dd^c
  \log|f'|^2$ is smooth. 
It follows that 
  \begin{equation}\label{beck}
    \pi^* u \1_{X\setminus \pi^{-1}Z} dd^c \pi^* u =
  C (-\log|f^0|^2-\gamma)^{2\epsilon-2} \frac{ d f^0 \w d \overline
    {f^0}}{|f^0|^2}
  +\beta,
    \end{equation}
where $C$ is a constant, $\gamma=\log |f'|^2$, and $\beta$ has locally
finite mass.  
Moreover, for $j>1$, $ \pi^* u \1_{X\setminus \pi^{-1}Z} (dd^c \pi^* u)^j$ is a sum of
terms that are integrable or of the form a smooth form times 
\begin{equation}\label{hemskt}
  (-\log|f^0|^2-\gamma)^{a}
  \frac{ d f^0 \w d \overline
    {f^0}}{|f^0|^2},
\end{equation}
where $a\leq 2\epsilon -2$. 
By Hironaka's theorem we may assume that $f^0$ is a monomial, and then
an elemenary computation yields that \eqref{hemskt} has locally
finite mass if and only if $a<-1$. 
Hence \eqref{beck}, and thus $u\la dd^c u\ra$, have locally finite mass
if and only if $\epsilon <1/2$. Moreover, if $\epsilon <1/2$, then
$u\la dd^c u \ra^j$ has locally finite
mass for $j>1$. 
We conclude that 
$u\in \G(\Omega)$ if and only if $u\in
\G_1(\Omega)$, which in turn holds if and only if 
$\epsilon <1/2$.

Finally, note that a necessary condition for $u+v$ to be in
$\G_k(\Omega)$ for any $k$ is that $v\la dd^c u\ra$ has locally finite
mass. Now 
the pullback of $v\la dd^c u\ra$ 
contains a term of the form 
\[
 C (-\log|f^0|^2-\gamma)^{2\epsilon-2} \frac{ d f^0 \w d \overline
    {f^0}}{|f^0|^2}, 
\]
where $C$ is a constant.
Since this does not have 
locally finite mass for any $\epsilon\in (0,1)$, it follows that $u+v\notin
\G_k(\Omega)$ for any such $\epsilon$. 
\end{ex}

Note in view of Example ~\ref{utvidgat} that, in contrast to the case
of $\D(\Omega)$, cf.\ \cite{B06}*{Theorem~1.2},  it is not true that 
$v\in \G_k(\Omega)$ and $v\leq u$ imply that $u\in \G_k(\Omega)$.

\begin{remark}\label{alexander}
Let $u$ and $v$ be as in Example ~\ref{utvidgat}. It is not hard to
check that $u+v$ has asymptotically analytic singularities in the
sense of Rashkovskii, \cite{R}*{Definition~3.4}. Indeed, this follows
after noting that for each $\delta>0$, there is a $C_\delta>0$ such
that
$(1+\delta)v -C_\delta \leq u+v\leq v$. We saw above that
$u+v\notin\G_k(\Omega)$ for any $k$. Hence we conclude that psh
functions with asymptotically analytic singularities are not in 
$\G_k(\Omega)$ in general. 
  \end{remark}

Note that $u$ in Example ~\ref{utvidgat} has sul; it is even locally
bounded outside the analytic variety $\{f=0\}$. 
Next, we will describe a way of constructing functions in $\G(X)$ that do
not have sul. 
We will use the following lemma that follows as in Example
~\ref{aladab}, see also \cite{LRSW}*{Proposition~3.2}.
\begin{lma}\label{gummistovel}
Assume that $u, v$ are psh functions with analytic singularities on
$X$. 
Then, for any smooth positive $(1,1)$-form $\omega$, test function
$\chi\geq 0$, and $i\leq j\leq n-1$,
\begin{equation*}
\int_X u \la dd^cu \ra^{i}  \w \la dd^c v\ra^{j-i}\w\chi \omega^{n-j}>-\infty. 
\end{equation*}
\end{lma}

\begin{ex}\label{plutt}
Let  $\U$ be a neighborhood of the unit ball $\B\subset \C^n$ and let 
 $v_i$, $i=1,2,\ldots$ be negative psh functions with analytic
singularities in $\U$. 
Let
$
  u_\ell=\sum_{i=1}^\ell b_i  v_i,
$
where $b_i>0$, and let
\[
u=\lim_{\ell\to\infty}  u_\ell=\sum_{i=1}^\infty b_i  v_i.
\]
We claim that we can choose $b_i$ so that the restriction of $u$ to
$\B$ is in $\G(\B)$.
Let $\chi$ be a smooth non-negative function with compact support in $\U$ such that
$\chi\equiv 1$ in $\B$, and let $\omega$ be a smooth strictly positive
$(1,1)$-form.  
It is enough to prove that, given $C>0$, we can
choose $b_i$ so that
\begin{equation}\label{skakad}
  \int_\U u_\ell \la dd^c u_\ell \ra^j \w \chi\omega^{n-j} >-C, ~~~~
  \ell \geq 1, ~~~~~~ j \leq n-1.
  \end{equation}
Since $v_i<0$, $u_\ell\searrow u$ and it follows from Lemma ~\ref{handduk} that 
 \begin{equation*}
   \int_\U u \la dd^c u  \ra^j \w \chi \omega^{n-j}>-C, ~~~~~~ j\leq n-1, 
 \end{equation*} 
 and thus $u\in \G(\U)$.

It remains to prove \eqref{skakad}.
Since \eqref{husqvarna} is
multilinear, it follows that
\begin{multline*}
 \int_\U u_\ell \la dd^c u_\ell \ra^j \w \xi \omega^{n-j}
  =
  \int_\U (u_{\ell-1}+b_\ell v_\ell)
  \big \la dd^c (u_{\ell-1}+b_\ell v_\ell)
  \big \ra^j \w \xi \omega^{n-j}
  =\\
  \int_\U u_{\ell-1} \la dd^c u_{\ell-1} \ra^j \w \xi \omega^{n-j}
  +
  \sum_{r=1}^{j+1} b_\ell^r T_r, 
  \end{multline*}
where each $T_r$ is a sum of terms of the form
\[
  \int_X \phi \la dd^c u_{\ell-1} \ra^\kappa \w \la dd^c v_\ell\ra^{j-\kappa} \w \xi \omega^{n-j},
\]
where $\phi=u_{\ell-1}$ or $\phi=v_\ell$; in particular they are
independent of the choice of $b_\ell$. By Lemma ~\ref{gummistovel} each
such integral is $>-\infty$. Thus by choosing $b_\ell$ small enough we
can make the difference between $\int_X u_\ell \la dd^c u_\ell \ra^j \w
\xi \omega^{n-j}$ and 
 $ \int_\U u_{\ell-1} \la dd^c u_{\ell-1} \ra^j \w \xi \omega^{n-j}$ 
arbirtrarily small. In particular, for any $C>0$ we can inductively
choose $b_i$ so that \eqref{skakad} holds. 
\end{ex}

Let us look at some explicit examples. 
\begin{ex}\label{racket}
  Given $a=(a_1,\ldots, a_k)\in (\C^n_x)^k$, let $|a\cdot x|^2=
  |a_1\cdot x|^2+\cdots + |a_k\cdot x|^2$, and let
 $v_a = \log |a\cdot x|^2$. Then $v_a$ is
  psh with analytic singularities and the unbounded locus of $v_a$
  equals
  $P_a:=\bigcap_{i=1}^k \{a_i\cdot x =0\}$.
  Note that for generic choices of $a$, $P_a$ is a plane of
  codimension $k$. 
  Next, choose $a^i \in (\C^n_x)^k, i=1,2,\ldots$ so that $\bigcup_i
  P_{a^i}$ is dense in $\C^n$, let
  $v_i=v_{a^i}$, and let
  $u=\sum b_i v_i$ be constructed as in Example
  ~\ref{plutt}. Then the restriction of $u$ to $\B$ is in $\G(\B)$, but $u$ is
  not locally bounded anywhere; in particular, $u$
  does not have sul.
  \end{ex}

\begin{ex}\label{punkt}
 As in Example ~\ref{plutt}, let $\U$ be a neighborhood of the unit
 ball $\B\subset \C^n_x$.  Given $a=(a_1, \ldots, a_n)\in \U$, let $v_a=\log
  |x-a|^2=\log (|x_1-a_1|^2+\cdots +|x_n-a_n|^2)$. Then $v_a$ is
  psh in $\U$ with analytic singularities and the unbounded locus of $v_a$ equals $a$.
  Next, let $a^i, i=1,2,\ldots$, be a dense subset of $\B$, let
  $v_i=v_{a^i}$, and let
  $u=\sum b_i v_i$ be constructed as in Example
  ~\ref{plutt}. Then the restriction of $u$ to $\B$ is in $\G(\B)$, but $u$ is
  not locally bounded anywhere; in particular, $u$
  does not have sul.
  In fact, $u\in \D(\U)$, see, e.g.,
  \cite{B06}*{Theorem~2}. 
  \end{ex}

\subsection{Direct products}\label{lokalproducerat}


On a direct product $X=X_1\times X_2$ we can produce new examples of
functions in $\G_k(X)$ by combining functions in $\G_k(X_1)$ and
$\G_k(X_2)$.

\begin{prop}\label{pengar}
  Assume that for $i=1,2$, $u^i\in \G_k(X_i)$, where $X_i$ is a
  complex manifold. 
Let
$X=X_1\times X_2$ and let 
$\pi_i\colon X\to X_i$, $i=1,2$, be the natural projections.
Then $u:=\pi^*_1u^1+\pi^*_2u^2\in \G_k(X)$. 
%
\end{prop}

\begin{proof}
  First note that $u\in \psh(X)$.
 For $i=1,2$, let $\omega_i$ be a smooth strictly positive $(1,1)$-form on
 $X_i$, so that $\omega:=\pi_1^*\omega_1+\pi_2^*\omega_2$ is a smooth 
 strictly positive $(1,1)$-form on $X$. 
To prove that $u\in \G_k(X)$ it suffices to prove that, for any function $\chi$ of the form
$\chi=\pi_1^*\chi_1\cdot \pi_2^*\chi_2$, where $\chi_i$ is a
non-negative test function 
in $X_i$, and for 
$0\leq j\leq k$,
\begin{equation*}
  -\infty < \int_X u\la dd^c u \ra^j \w
  \chi \omega^{n-j}=
  \lim_{\ell\to \infty} 
 \int_{O_\ell} u\la dd^c u \ra^j \w
  \chi \omega^{n-j},  
\end{equation*}
where $O_\ell=\{u>-\ell\}\subset X$.
Note that for $\lambda$ large enough, $O_\ell\cap \supp \chi \subset O^1_\lambda\times
O^2_\lambda$, where $O^i_\lambda=\{u_i>-\lambda\}\subset X_i$. 
In particular, $u=\pi_1^*u^1_\lambda + \pi_2^* u^2_\lambda$
in $O_\ell\cap \supp \chi$, where $u^i_\lambda=\max (u^i, -\lambda)$. Thus, since 
the non-pluripolar Monge-Amp\`ere operator  is local in the plurifine topology, 
\begin{multline}\label{borsta}
  \int_{O_\ell} u\la dd^c u\ra^j \w
   \chi\omega^{n-j}
   =\\
   \int_{O_\ell} (\pi_1^*u^1_\lambda + \pi_2^* u^2_\lambda)
   \big ( dd^c (\pi_1^*u^1_\lambda + \pi_2^* u^2_\lambda) \big )^j \w
   \pi_1^*\chi_1\pi_2^*\chi_2
(\pi_1^*\omega_1+\pi_2^* \omega_2)^{n-j}. 
\end{multline}
Since \eqref{mixad} is multilinear, \eqref{borsta} is a finite sum of terms of the form
\begin{multline}\label{kurra}
  \int_{O_\ell} \pi_1^*u^1
( dd^c \pi_1^*u^1_\lambda )^{j_1} \w \pi_1^*\chi_1 (\pi_1^*\omega_1)^{n_1-j_1}\w
 ( dd^c \pi_2^* u^2_\lambda )^{j_2} \w 
 \pi_2^*\chi_2 (\pi_2^* \omega_2)^{n_2-j_2}
 \geq \\
 \int_{O^1_\lambda} u^1
( dd^c u^1_\lambda)^{j_1} \w \chi_1\omega_1^{n_1-j_1}\int_{O^2_{\lambda}}
 (dd^c u^2_\lambda)^{j_2} \w 
 \chi_2 \omega_2^{n_2-j_2} 
  \end{multline}
or of the form where the first factor $\pi^*_1u^1$ is replaced by
$\pi^*_2u^2$. Since $u^i\in \G_k(X_i)$ each factor in the
right hand side of \eqref{kurra} is bounded uniformly in $\lambda$. 
\end{proof}


\begin{ex}
  Let $\B\subset \C^n$ be the unit ball.
Choose $u_1\in \D(\B)\cap \G(\B)$ that does
not have analytic singularities, e.g., let $u_1$ be as in Example
~\ref{punkt}. Moreover let
$u_2$ be a psh function in $\B$ with analytic singularities that is not in
$\D(\B)$. Then by Proposition
~\ref{lokalproducerat}, 
$u=\pi_1^*u^1+\pi_2^* u^2\in \G(\B\times \B)$.
Now $u$ neither has analytic singularities nor is in $\D(\B\times
\B)$.
  \end{ex}


  \section{Regularization}\label{regsec}

  We will prove Theorem ~\ref{vegetera}. In fact, we will prove the
  following slightly more general result.
  \begin{thm}\label{vegetarisk}
  Assume that $u\in \G_k(X)$ and that $v$ is a smooth psh function on $X$. 
Then, for $p\leq k+1$, 
  \begin{equation}\label{knippe}
    \big(dd^c \max(u,v-\ell)\big)^p \to [dd^c u]^p + \sum_{j=1}^{p-1}
    \S_j(u)\w (dd^c v)^{p-j}, \quad \ell\to\infty.  
  \end{equation}

More generally, let $\chi_\ell:\R\to \R$ be a sequence of nondecreasing
convex functions, bounded from below, 
that decreases to $t$
as $\ell\to \infty$, and let $u_\ell=\chi_\ell\circ (u-v)+v$.  Then,
for $p\leq k+1$, 
\begin{equation}\label{vandrare}
   (dd^cu_\ell)^p
  \to [dd^c u]^p + \sum_{j=1}^{p-1}
 \S_j(u)\w  (dd^c v)^{p-j}, \quad \ell\to\infty.  
    \end{equation}
  \end{thm}

  To illustrate the idea of the proof, let us start with a special
  case.
  \begin{proof}[Proof of \eqref{knippe} when $v=0$]
    Let $u_\ell=\max (u, -\ell)$. It is enough to prove that
    \begin{equation}\label{andning}
      (dd^c u_\ell)^p = dd^c u_\ell \w \la dd^cu \ra^{p-1}
\end{equation}
since the right hand side converges to $[dd^c u]^p$.

To prove \eqref{andning}, let $\xi$ be a test form on $X$ and consider
\begin{equation}\label{kiwi}
\int_X u_{\ell} \big ((dd^c u_\ell)^{p-1} - \la dd^c u\ra^{p-1} \big ) \w dd^c \xi.
  \end{equation}
Since $u_{\ell}=u$
in $O_\ell:=\{u>-\ell\}$ and the non-pluripolar Monge-Amp\`ere operator is local in the plurifine topology, we
get that $ u_{\ell} (dd^c u_\ell)^{p-1} = u_\ell \la dd^c u\ra^{p-1}$
in $O_\ell$. Using that $u_\ell=-\ell$ in $X\setminus O_\ell$ we see that
\eqref{kiwi} equals
\begin{equation*}
-\ell \int_X \big ((dd^c u_\ell)^{p-1} - \la dd^c u\ra^{p-1} \big ) \w
dd^c \xi,
  \end{equation*}
  which vanishes by Stokes' theorem. Thus \eqref{andning} follows. 
\end{proof} 
The proof of \eqref{knippe} for general $v$ follows in the same way after replacing 
\eqref{andning} by Lemma ~\ref{dynamit} below (with $\ell_1=\cdots
=\ell_p=\ell$). The general case follows by writing $\chi$ as a
superposition of functions $\max(t, -s)$. 
For this we need some
  auxiliary results. Let us first consider an elementary lemma.

  \begin{lma}\label{pam}
Assume that  $\chi$ is nondecreasing, convex, and bounded on $(-\infty,0]$ and that
$\chi(0)=0$ and  $\chi'(0)=1$.  Let $g(s)=\chi''(-s)$.
Then $g(s)ds$ is a probability measure on $[0, \infty)$. 
Moreover, 
\begin{equation*}
\chi(t)=\int_{s=0}^\infty \max(t, -s) g(s)ds.
\end{equation*}
\end{lma}

Here $\chi'$ should be interpreted as the left derivative of $\chi$, which is
always well-defined since $\chi$ is convex. 

\begin{proof}
Note that $g$ is a (positive) measure since $\chi$ is
convex.

First assume that $\chi$ is smooth.  Notice that $\chi'(-s)\to 0$ when
$s\to \infty$ since $\chi$ is bounded.
Therefore
\begin{equation*}
\int_{0}^\infty g(s)ds = \int_{0}^\infty\chi''(-s)ds=-\chi'(-s)\Big|_0^\infty=\chi'(0)=1
\end{equation*}
and thus $g$ is a probability measure. 
Moreover,
\begin{multline*}
  \int_{0}^\infty \max(t, -s)\chi''(-s)ds=
  -\max(t, -s)\chi'(-s)\Big|_0^\infty+\\
  \int_0^\infty\frac{d}{ds}\max(t, -s)\cdot \chi'(-s)ds =
-\int_0^{-t}\chi'(-s)ds=\chi(t).
\end{multline*}


If $\chi$ is not smooth, the above
arguments goes through verbatim if we understand $\chi'$ as the left
derivative of $\chi$, $g(s)ds$ as the corresponding Lebesque-Stieltjes
measure, and the integrals as Lebesque-Stieltjes integrals.

%
%
\end{proof}

Assume that $T(s)$, $s\in \R^p_{\geq 0},$ is a continuous family of
currents of order zero on $X$ and $G(s)ds$ is a measure on $\R^p_{\geq 0}$. Then
$\int_{\R_{\geq 0}^p}  T(s) G(s)ds$ is a well-defined current of order
zero on $X$, defined by
\[
  \int_X\int_{\R_{\geq 0}^p}  T(s) G(s) ds \w \xi = \int_{\R_{\geq
      0}^p}  \int_X T(s) \w \xi G(s) ds,  
\]
if $\xi$ is a (continuous) test form on $X$.

\begin{lma}\label{fyrkant} 
Assume that $T(t), t\geq 0$, is a continuous family of positive currents that
converges weakly to a current $T_\infty$ on $X$ when $t$ tends to
$\infty$.
Moreover assume that $G_\ell(s)=
G_\ell(s_1,\ldots, s_p)$ is a sequence of
probability measures on $\R^p_{\geq 0}$ such that for all $R>0$, 
\begin{equation}\label{havtorn}
  \lim_{\ell\to \infty} \int_{s_p>R}\cdots\int_{s_1>R}
  G_\ell(s) ds=1,
\end{equation}
where $ds=ds_1\cdots ds_p$.
Finally, assume that $\rho:\R_{\geq 0}^p\to \R$ is a continuous function such
that $\rho(s)\geq \min_j s_j$. 
Then 
\begin{equation*}
\int_{\R_{\geq 0}^p} T\big (\rho(s)\big ) G_\ell(s) ds
\to  T_\infty
\end{equation*}
weakly when $\ell\to \infty$. 
\end{lma}

\begin{proof}
  Let $\xi$ be a test form on $X$.
  We need to prove that
  \begin{equation}\label{simma}
    \int_X\int_{\R_{\geq 0}^p}  T\big (\rho(s)\big ) G_\ell(s) ds \wedge \xi = 
 \int_{\R_{\geq 0}^p}  \int_XT\big (\rho(s)\big ) \wedge \xi  G_\ell(s) ds
    \to
    \int_X T_\infty \wedge \xi
    \end{equation}
    when $\ell\to \infty$.

    Take $\epsilon>0$. Then there is an $R>0$ such that for $t>R$,
    \begin{equation*}
      \big | \int_X \big (T(t) -T_\infty\big ) \wedge \xi \big |
      <\epsilon.
    \end{equation*}
    Let $A_R=\{s_j>R, j=1,\ldots, p\}$. Then $\rho(s)>R$ on $A_R$ and, 
    since $G_\ell$ are probability measures, it follows that
      \begin{equation}\label{stockholm} 
      \big | \int_{A_R} \int_X\big (T\big(\rho(s) \big) -T_\infty\big )
      \wedge \xi G_\ell (s) ds\big |
      <\epsilon
    \end{equation}
    for all $\ell$.
    Since $T(t)$ is continuous and convergent, there is
    an $M\in \R$ such that 
        $|\int_X T(t) \wedge \xi |<M$ for all
   $t$. 
   Now, by \eqref{havtorn}, 
    \begin{equation}\label{oslo}
      \big | \int_{\R^p_{\geq 0}\setminus A_R}\int_X \big (T\big(\rho(s)\big) -T_\infty\big )
      \wedge \xi  G_\ell (s) ds \big |
      < 2M \int_{\R^p_{\geq 0}\setminus A_R} G_\ell(s) \to 0
    \end{equation}
    when $\ell\to \infty$.
     Since $\epsilon$ is arbitrary, \eqref{simma} follows from
  \eqref{stockholm} and \eqref{oslo}. 
  \end{proof}

\begin{lma}\label{dynamit}
Assume that $u\in \G_{k}(X)$  
and that $v$ is a smooth psh function on
$X$. For $\ell \geq 0$, let $u_\ell=\max (u, v-\ell)$. Then for any
$p\leq k+1$ and any 
$\ell_1,\ldots, \ell_p$ such that 
$\ell_1\geq \cdots \geq \ell_p$, 
\begin{multline}\label{tobin} 
  dd^c u_{\ell_p}\w\cdots\w dd^c u_{\ell_1} =\\
  dd^c u_{\ell_p}\w \la dd^c u\ra^{p-1}  +\sum_{j=1}^{p-1}\big (dd^c u_{\ell_j} -
  \la dd^c u\ra \big ) \w \la dd^c u \ra^{j-1}\w (dd^c v)^{p-j}. 
    \end{multline}
\end{lma}

\begin{proof}
  We claim that for $j=1,\ldots, p-1$, 
  \begin{multline}\label{gasp}
   dd^c u_{\ell_{j+1}}\w\cdots\w dd^c u_{\ell_1} \w  (dd^c v)^{p-j-1} =
    dd^c u_{\ell_{j}}\w\cdots\w dd^c u_{\ell_1}\w (dd^c v)^{p-j} 
    +\\
      dd^c u_{\ell_{j+1}}\w\la dd^c u\ra^j \w (dd^c v)^{p-j-1} 
      -  \la dd^c u\ra^j\w (dd^c v)^{p-j}. 
    \end{multline}
Taking \eqref{gasp} for granted, by recursively
applying it to $j=p-1, \ldots, 1$ we obtain ~\eqref{tobin}.

To prove \eqref{gasp}, let $\xi$ be a test form on $X$ and consider
\begin{equation}\label{sax}
\int_X u_{\ell_{j+1}} \big (dd^c u_{\ell_j}\w\cdots\w
dd^c u_{\ell_1}\w (dd^c v)^{p-j-1} - \la dd^c u\ra^j \w (dd^c v)^{p-j-1} \big ) \w dd^c \xi.
  \end{equation}
Since $u_{\ell_j}=\cdots = u_{\ell_1}=u$
in $O=\{u>v-\ell_{j+1}\}$, which is open in the plurifine topology,
and since the non-pluripolar Monge-Amp\`ere operator is local in the plurifine topology, we
get that 
$$
dd^c u_{\ell_j}\w\cdots\w dd^c u_{\ell_1}\w (dd^c v)^{p-j-1} = \la
dd^c u\ra^j\w (dd^c
v)^{p-j-1}$$ there.
Using that  $u_{\ell_{j+1}}=v-\ell_{j+1}$ in
$X\setminus  O$,
we see that \eqref{sax} equals
\begin{equation}\label{laxen}
\int_X (v-\ell_{j+1})\big (dd^c u_{\ell_j}\w\cdots\w
dd^c u_{\ell_1}\w (dd^c v)^{p-j-1} - \la dd^c u\ra^j \w (dd^c v)^{p-j-1} \big ) \w dd^c \xi.
  \end{equation}
Now \eqref{gasp} follows from \eqref{sax} and \eqref{laxen} by Stokes'
theorem, since $\la dd^c u\ra^j$ is closed. 
  \end{proof}

  \begin{proof}[Proof of Theorem ~\ref{vegetarisk}]
    Since \eqref{vandrare}  is a local statement we may assume that $u-v$
    is bounded from above. In fact, we may assume that $u-v< 0$. Otherwise,  
    if $u-v< c$, let $\check u=u-c$ and $\check
    \chi_\ell(t)=\chi_\ell(t+c)-c$. Then $\check u-v< 0$ and $\check
    \chi_\ell$ is a seqence of functions as in the assumption of the
    theorem. Moreover $[dd^c\check u]^k = [dd^c u]^k$
    and 
$\check \chi_\ell \circ (\check u - v) = \chi_\ell \circ (u-v) -c$; in particular
$dd^c (\check \chi_\ell \circ (\check u - v) )= dd^c (\chi_\ell
\circ (u - v))$. Thus it suffices to prove \eqref{vandrare} for $\check u$ and $\check
\chi_\ell$. 

\smallskip 

Throughout this proof let $u_\ell=\max (u, v-\ell)$ and $\tilde
u_\ell=\chi_\ell\circ (u-v) + v$.
Let us first assume that $\chi_\ell(0)=0$ and $\chi'_\ell(0)=1$ so
that $\chi_\ell$ (restricted to $(-\infty, 0]$) is as in Lemma ~\ref{pam}. 
Let $g_\ell(t)=\chi_\ell''( -t)$. 
Note that $u_\ell=\max (u-v, -\ell) + v$ and thus by Lemma ~\ref{pam}  
\begin{equation*}
  \int_{s=0}^\infty u_sg_\ell(s)ds =
  \int_{s=0}^\infty \max (u-v, -s) g_\ell(s)ds +
  \int_{s=0}^\infty v g_\ell(s)ds = \chi_\ell\circ (u-v) + v=\tilde u_\ell.
\end{equation*}
It follows from Lemma ~\ref{pronto} that $dd^c u_{s_p}\w\cdots\w
dd^c u_{s_1} $ is continuous in $s$. Thus 
\begin{equation}\label{himmelrike}
  (dd^c \tilde u_\ell)^p = 
  \int_{s_p=0}^\infty\cdots\int_{s_1=0}^\infty dd^c u_{s_p}\w\cdots\w
dd^c u_{s_1} g_\ell(s_1)\cdots g_\ell(s_p) ds.
  \end{equation}
Let $\rho_j:\R^p_{\geq 0}\to \R$ be the function that maps
$s=(s_1,\ldots, s_p)$ to the $j$th largest $s_i$; in particular,
$\rho_p(s)=\min_i s_i$. 
Since $dd^c u_{s_p}\w\cdots\w
dd^c u_{s_1}$ is commutative in the factors $dd^cu_{s_i}$, it follows from Lemma ~\ref{dynamit} that
\begin{equation*}
  dd^c u_{s_p}\w\cdots\w dd^c u_{s_1} =
  \sum_{j=1}^{p} dd^c u_{\rho_j(s)} \w \la dd^c u \ra^{j-1} \w (dd^c v)^{p-j} -
  \sum_{j=1}^{p-1}\la dd^c u \ra^{j}\w (dd^c v)^{p-j}. 
  \end{equation*}

For $j=1,\ldots, p$, let $T_j(t)=dd^c u_{t} \w \la dd^c
u \ra^{j-1}\w (dd^c v)^{p-j}$. By Lemma ~\ref{konto}, $T_j(t)$ is continuous
in $t$. Moreover, since $u\in \G_p(X)$, 
it converges weakly to $[dd^c u]^j\w (dd^c v)^{p-j}$.   
By Lemma ~\ref{pam}, $G_\ell(s):=g_\ell(s_1)\cdots g_\ell(s_p)ds$ is a probability
measure. Since $\chi_\ell(t)\to t$ when $\ell\to\infty$, given $R>0$,
$\chi_\ell'(-R)\to 1$ and thus $\int_{-R}^0g_\ell(s)ds \to 0$. It follows that $G_\ell(s)$ satisfies \eqref{havtorn}. Since $\rho_j$ is
continuous and $\rho_j(s)\geq
\min_i s_i$, Lemma ~\ref{fyrkant} yields that
\begin{equation*}
 \lim_{\ell\to\infty} \int_{\R^p_{\geq 0}}
dd^c u_{\rho_j(s)} \w \la dd^c u \ra^{j-1}\w (dd^c v)^{p-j} 
G_\ell(s)ds = 
[dd^c u]^j\w (dd^c v)^{p-j}.
\end{equation*} 
Hence, since $G_\ell$ is a probability measure, the limit of
\eqref{himmelrike} when $\ell\to \infty$ 
equals
\begin{equation*}
\sum_{j=1}^{p} [dd^c u]^j \w (dd^c v)^{p-j} -
\sum_{j=1}^{p-1} \la dd^c u \ra^{j}\w (dd^c v)^{p-j}
=
[dd^c u]^p+ \sum_{j=1}^{p-1} \S_j(u)\w (dd^c v)^{p-j}.
\end{equation*}

Finally, let us consider a sequence $\chi_\ell$ where we drop the
extra assumptions on $\chi_\ell(0)$ and $\chi_\ell'(0)$.
Since $\chi_\ell(t)$ are convex functions converging to $t$,
$\chi_\ell'(0)\to 1$; in particular $\chi_\ell'(0)\neq 0$ for large
enough $\ell$. 
Let $\hat \chi_\ell = (\chi_\ell-\chi_\ell(0)
)/\chi_\ell'(0)$. Then $\hat\chi_\ell$ is a sequence of nondecreasing
convex functions bounded from below, such that $\hat\chi_\ell(t)\to t$,
when $\ell\to \infty$, and $\hat \chi_\ell(0)=0$ and $\hat
\chi_\ell'(0)=1$. By the above arguments 
$$
(dd^c \hat u_\ell)^k \to  [dd^c u]^k+ \sum_{j=1}^{k-1}\S_j(u)\w (dd^c
v)^{k-j} 
$$
for $k\leq p$, where $\hat u_\ell = \hat
\chi_\ell\circ (u-v)+v$. 
Note that 
$dd^c \tilde u_\ell = \chi_\ell'(0) dd^c \hat
u_\ell+(1-\chi_\ell'(0))dd^c v$. Since $\chi_\ell'(0)\to 1$, it
follows that 
$$
\lim_{\ell\to \infty} (dd^c \tilde u_\ell)^p=\lim_{\ell\to \infty}
(dd^c \hat u_\ell)^p = [dd^c u]^p+ \sum_{j=1}^{p-1}
\S_j(u)\w (dd^c v)^{p-j}.
$$
\end{proof}

\begin{remark} 
  Note that the proof above only uses that $\chi_\ell(t)\to t$, when
  $\ell
  \to \infty$. Therefore we could, in fact, drop the assumption that
  $\chi_\ell$ is a decreasing sequence in Theorem ~\ref{vegetarisk}
  (as well as in the theorems in the introduction). 
\end{remark}

\begin{remark}\label{flertal}
   It follows from the proof above that
  we can choose different
sequences $\chi_\ell$ in Theorem ~\ref{vegetarisk} (and the theorems
in the introduction) 
and get the following generalization: 
\emph{For
$\lambda=1,\ldots, p$, let $\chi^\lambda_\ell:\R\to \R$ be a sequence of nondecreasing
convex functions, bounded from below, 
that decreases to $t$
as $\ell\to \infty$, and let $u^\lambda_\ell=\chi^\lambda_\ell\circ (u-v)+v$.  Then
\begin{equation*}
dd^c u_\ell^p \w\cdots\w dd^c u_\ell^1   
  \to [dd^c u]^p + \sum_{j=1}^{p-1}
  \S_j(u) \w (dd^c v)^{p-j}, \quad \ell\to\infty.  
    \end{equation*}}
Indeed, the proof goes through verbatim if we let
$g_\ell^\lambda(t)=(\chi_\ell^\lambda)''(-t)$ and $G_\ell(s)=g^1_\ell(s_1)\cdots
g_\ell^p(s_p)ds$. 
\end{remark}

\begin{remark}\label{oversattning}
     Let us relate Theorem ~\ref{vegetarisk} to \cite{B}*{Theorem~1}. 
Note that 
\begin{equation*}
  u_\ell=\chi_\ell\circ (u-v) + v
  =: \varphi_\ell +v,
\end{equation*} 
where $\varphi_\ell$ is a sequence converging to the
quasiplurisubharmonic (qpsh)
function $\varphi:= u-v$. 
Theorem ~1 in \cite{B} asserts that if $\varphi$ has analytic
singularities, then 
\begin{equation}\label{tennis}
  (dd^c \varphi_\ell)^p\to [dd^c\varphi]^p, 
\end{equation}
where $[dd^c\varphi]^n$ is an extension of \eqref{pig} to qpsh
functions, see, e.g., \cite{B, LRSW} for details. 
Using \eqref{tennis}, we see that 
\begin{equation*}
 (dd^c u_\ell)^p = (dd^c \varphi_\ell+dd^c v)^p\to ([dd^c\varphi] + dd^cv)^p. 
\end{equation*}
It follows from the definition of $[dd^c \varphi]^p$, cf.\ \eqref{pig}
that, in fact, $([dd^c\varphi] + dd^cv)^p$ equals the right hand side
of \eqref{vandrare} (e.g., by arguments as in \cite{LRSW}).
Thus if $u$ has analytic singularities Theorem ~\ref{vegetarisk} follows
from \eqref{tennis}.
\end{remark}

\begin{ex}\label{storoliten}
  Let $u=\log|z_1|^2+|z_2|^2$ in the unit ball $\B$ in $\C^2$. Then $u$ has analytic
  singularities and thus $u\in \G(\B)$.
  By the Poincar\'e-Lelong formula $$[dd^cu]=[z_1=0]+ dd^c |z_2|^2$$ and 
  one easily checks that
\[
   [ dd^c u]^2 = [z_1=0]\w dd^c |z_2|^2 \neq 0.
 \]
Since $S_1(u)=[z_1=0]\neq 0$, we know from Theorem ~\ref{vegetarisk}
that there are sequences of bounded psh
functions $u_\ell$ decreasing to $u$ such that the limits of the
Monge-Amp\`ere currents $(dd^c u_\ell)^2$ converge to different
measures. In fact, it follows that we can find $u_\ell$ so that the
mass of the measures are arbitrarily large: 
Let $v_\ell=\ell|z_2|^2$ and let $u_{\ell, \lambda}=\max(u, v_\ell-\lambda)$. By Theorem
~\ref{vegetarisk}, 
\begin{equation}\label{amazer}
  \lim_{\lambda \to \infty} (dd^c u_{\ell, \lambda})^2 = [dd^c u]^2+
  \ell [z_1=0]\w dd^c |z_2|^2.
\end{equation}
If we choose $\lambda_1 \ll \lambda_2 \ll \lambda_3 \ll\ldots$,
then
$u_\ell:=u_{\ell, \lambda_\ell}$ is a sequence of bounded psh
functions decreasing to $u$ and in view of \eqref{amazer} $(dd^c u_\ell)^2$ do
not have locally uniformly bounded mass.

In fact, the function $u$ is a maximal psh function and therefore in
this case it is 
possible to find a sequence of smooth psh functions $u_\ell$
decreasing to $u$ such that $(dd^c u_\ell)^2 $ converges
weakly to $0$, see \cite{ABW}*{Example~3.4}. 
\end{ex}

The following example shows that one needs some condition on a psh
function $u$ for the Monge-Amp\`ere currents of the natural regularization $u_\ell=\max (u, -\ell)$
to converge. In particular, $u$ below is an example of a psh
function that does not have locally finite non-pluripolar energy.

\begin{ex}\label{varannan}
Consider the plurisubharmonic function $$u(z, w)=\sup_{k\geq
  1}^*\big \{(1+1/k)\log|z|^2-a_k+(1-(-1)^k)|w|^2\big \}$$ in the bidisc
$\Di\times \Di$ for some choice of $a_k>0$, $k=1,2,\ldots$; here $*$
denotes the usc regularization. 
It is not hard to see that if we choose $0 \ll a_0 \ll a_1
\ll a_2 \ll \ldots$, then there is an increasing sequence $\ell_k\in
\N$ such that $u(z,w)=(1+1/2k)\log|z|^2-a_{2k}$ on the set
$\{|u(z,w)+\ell_{2k}|<1\}$, whereas $u(z,w)=(1+1/(2k+1))\log|z|^2-a_{2k+1}+2|w|^2$ on the set
$\{|u(z,w)+\ell_{2k+1}|<1\}$.
Note that 
\begin{equation*}
  \big (dd^c \max((1+1/k)\log|z|^2-a_k,-\ell) \big )^2=0
  \end{equation*}
  while
  \begin{equation*}
    \big (dd^c \max((1+1/k)\log|z|^2-a_k+2|w|^2,-\ell)\big
    )^2\to
     2(1+1/k)[z=0]\wedge dd^c|w|^2
  \end{equation*}
  as $\ell\to \infty$.
    It follows
that $$\lim_{k\to \infty} \big (dd^c \max(u,-\ell_{2k})\big )^2 =
\langle dd^c u\rangle^2,$$
whereas
$$
\lim_{k\to \infty} \big ( dd^c \max(u,
-\ell_{2k+1})\big)^2=
\langle dd^c u\rangle^2+2[z=0]\wedge
dd^c |w|^2.
$$
\end{ex}

\section{Global Monge-Amp\`ere products}\label{globala} 

Let $(X, \omega)$ be a compact K\"ahler manifold.
To define global analogues of the classes $\G_k(\Omega)$, let us first
recall some results
on the non-pluripolar Monge-Amp\`ere operator. 


Assume that
for $j=1, \ldots, n$, $\omega_j$ is a K\"ahler form on $X$ and $\varphi_j$ is
$\omega_j$-psh. 
Since \eqref{husqvarna} only depends on the currents $dd^c u_j$, 
\begin{equation}\label{jonkoping}
  \la dd^c \varphi_p+\omega_p \ra \w \cdots  \w \la dd^c \varphi_1+\omega_1 \ra,
\end{equation}
locally defined as
 $ \la dd^c (\varphi_p+h_p) \ra \w \cdots  \w \la dd^c
 (\varphi_1+h_1) \ra, $
where $h_j$ are local $dd^c$-potentials of the $\omega_j$, is a global
closed positive current on $X$, see Section ~\ref{pruttkudde}.

Assume that $\varphi_1, \ldots, \varphi_n\in\psh(X, \omega)$. Then, by
Stokes'  theorem,  
\begin{equation}\label{ringar}
   \int_X (dd^c \varphi_n+\omega) \w \cdots \w (dd^c \varphi_1 + \omega) 
  = \int_X \omega^{n},
  \end{equation}
  cf.\ \eqref{masslikhet}.
  %
For the non-pluripolar products we have the following monotonicity
property. Recall that a
function $\varphi$ on $X$ is quasiplurisubharmonic (qpsh) if it is
locally the sum of a psh function and a smooth function.
\begin{prop}\label{monoton}
Assume that $\varphi_1,\ldots, \varphi_n$ and $\psi_1,\ldots, \psi_n$
are $\omega$-psh functions, such that $\varphi_j\succeq \psi_j$. 
Then
\begin{equation*}
  \int_X \la dd^c \varphi_n+\omega \ra \w \cdots \w \la  dd^c
  \varphi_1+\omega \ra
  \geq
  \int_X \la dd^c \psi_n+\omega \ra \w \cdots \w \la dd^c
  \psi_1+\omega \ra
\end{equation*}
\end{prop}
As a consequence, the integral of \eqref{jonkoping} only depends on the singularity types
of the $\varphi_j$.
For 
$\omega$-psh functions with sul Proposition ~\ref{monoton} was proved in
\cite{BEGZ}*{Theorem~1.16}, in the case when
$\varphi_j=\varphi$ and $\psi_j=\psi$ in
\cite{WN}*{Theorem~1.2}, and in the general case in \cite{DDL}*{Theorem~1.1}.
Also, see \cite{Duc21}*{Theorem~1.1} for an even stronger monotonicity result.

We will use the following integration by parts result. 

\begin{prop}[\cite{BEGZ}*{Theorem~1.14}]\label{partial}
Let $A\subset X$ be a closed complete pluripolar set, and let $T$ be a
closed positive $(n-1,n-1)$-current on $X$. Let $\varphi_i$ and
$\psi_i$, $i=1,2$, be qpsh functions on $X$ that are locally bounded on
$X\setminus A$. If $u:=\varphi_1-\varphi_2$ and $v:=\psi_1-\psi_2$ are
globally bounded on $X$, then
\begin{equation}\label{ostbage}
  \int_{X\setminus A} udd^c v \w T =
  \int_{X\setminus A} vdd^c u \w T =
   -\int_{X\setminus A} dv\w d^c u \w T. 
  \end{equation}
  \end{prop}
  \begin{remark}\label{popcorn} 
In particular, the integrals in \eqref{ostbage} are well-defined, cf.\
Lemma~1.15 and the discussion after Theorem 1.14 in \cite{BEGZ}. Note
that if $v=u$, then \eqref{ostbage} is
non-positive.
    \end{remark}
    \begin{remark}\label{objektiv}
      Note that Proposition \ref{partial} 
    recently has been
generalized to the case when $\varphi_i$ and $\psi_i$ do not
necessarily have sul, see \cite{Xia19}*{Theorem 1.1} and \cite{Duc20}*{Theorem~2.6}.
\end{remark}

\subsection{The classes $\G_k(X, \omega)$}\label{vattenflaska}

The classes $\G_k(X)$ in Definition ~\ref{prisen}, are naturally
carried over to the global setting. Recall from the introduction that
on $(X, \omega)$, the 
non-pluripolar Monge-Amp\`ere currents $\la dd^c \varphi +
\omega\ra^j$ are always 
finite.

\begin{df}\label{prinsessan}
 Let $(X, \omega)$ be a compact K\"ahler manifold of dimension $n$.
 For $1\leq k\leq n-1$, we say that $\varphi\in \psh(X, \omega)$ has
\emph{finite non-pluripolar energy of order k}, $\varphi\in\G_k(X, \omega)$,
if, for each $1\leq j\leq k$,  
$\varphi$ is integrable with respect to $\la dd^c \varphi+ \omega \ra^j$.
\end{df}
Note that
$\varphi\in \psh(X, \omega)$ is in $\G_k(X, \omega)$ if and only if
$\varphi$ is integrable with respect ~to 
\[
  \sum_{j=0}^k\la dd^c\varphi + \omega \ra^j \w \omega^{n-j},
  \]
  cf.\ \eqref{grizzly}. 
Clearly
\begin{equation*}
\G_1(X, \omega)\supset \G_2(X, \omega) \supset \cdots\supset
\G_{n-1}(X, \omega)=\G(X, \omega), 
  \end{equation*}
where $\G(X, \omega)$ is as in Definition ~\ref{drottning}. 

If $\varphi\in \G_k(X, \omega)$, for $p=1,\ldots, k+1$, we can define currents
$$
[dd^c\varphi + \omega]^p=[dd^c (\varphi+h)]^p,
$$
where $h$ is a local potential for $\omega$, and 
$$
\S_p^\omega(\varphi)= [dd^c\varphi + \omega]^p -\la dd^c\varphi +
\omega\ra ^p
$$
as 
in Definition~\ref{prins}.
By Propositions ~\ref{slutna} and ~\ref{obero}, they are well-defined
global closed positive currents on $X$ that only depend on the current $dd^c
\varphi+\omega$ and not on the choice of $\omega$ as a K\"ahler
representative in the class $[\omega]$.

\begin{remark}
  Assume that $\varphi$ has sul and that $A$ is a closed complete pluripolar set such that
$\varphi$ is locally bounded in $X\setminus A$. Then 
$
\1_{X\setminus A} [dd^c \varphi+\omega]^p=\la dd^c \varphi+\omega\ra^p
$
so that 
$
\S_p^\omega(\varphi)=\1_{A} [dd^c \varphi+\omega]^p.
$
\end{remark}

From the definitions above and the basic properties of the
Monge-Amp\`ere currents we get an immediate proof of the mass formula
Theorem ~\ref{mannagryn}.
In fact, we prove the following slightly more general
version. 

\begin{thm}\label{mannagrynsgrot}
Assume that $\varphi\in
\G_{k}(X, \omega)$. Then for $p\leq k+1$, 
\begin{equation}\label{klassforestandare}
\int_X \la dd^c \varphi+\omega\ra^p\w \omega^{n-p}+ \sum_{j=1}^p \int_X
\S_j^\omega(\varphi)\w \omega^{n-j}=\int_X \omega^n.
\end{equation}
\end{thm}

\begin{proof}
First, note in view of Proposition ~\ref{obero} that, for $1\le j\le k$,
\begin{equation}\label{gorilla}
  dd^c \big (\varphi \la dd^c \varphi+\omega\ra^{j-1}\big ):= [dd^c
  \varphi+\omega]^j-\la dd^c \varphi+\omega\ra^{j-1}\w \omega
\end{equation}
is a well-defined exact current on $X$.  We claim that for $1\le j\le k$ we have
\begin{equation}\label{apa3}
\int_X\la dd^c \varphi+\omega\ra^j\w \omega^{n-j}- 
\int_X\la dd^c \varphi+\omega\ra^{j-1}\w \omega^{n-j+1}=
-\int_X \S_j^\omega(\varphi)\w \omega^{n-j}.
\end{equation}
In fact,
\begin{multline*}
\int_X\la dd^c \varphi+\omega\ra^{j}\w \omega^{n-j}+
\int_X \S_j^\omega(\varphi) \w \omega^{n-j}=
\int_X  [dd^c \varphi+\omega]^j\w \omega^{n-j}= \\
\int_X dd^c (\varphi \la dd^c \varphi+\omega\ra^{j-1})\w \omega^{n-j}
+\int_X \la dd^c \varphi+\omega\ra^{j-1}\w \omega^{n-j+1}
=\int_X \la dd^c \varphi+\omega\ra^{j-1}\w \omega^{n-j+1},
\end{multline*}
where we have used \eqref{gorilla} for the second equality and the last equality follows from Stokes' theorem. Thus \eqref{apa3} holds,
and summing from $1$ to $k$ we get \eqref{klassforestandare}. 
\end{proof}

Theorem ~\ref{mannagrynsgrot}  also follows immediately from
\eqref{ringar} and the
following slightly generalized version of Theorem ~\ref{kyrka}. 
\begin{thm}\label{kyrkbacke}
   Assume that $\varphi\in \G_k(X, \omega)$ and that $\eta$ is a
  K\"ahler form in $[\omega]$
  so that $\eta=\omega+dd^c
  g$, where $g$ is a smooth function on $X$.
Let $\varphi_\ell=\max (\varphi-g, -\ell)+g$.
  Then, for $1\leq p\leq k+1$, 
  \begin{equation}\label{krokus}
 (dd^c \varphi_\ell +\omega)^p \to  
[dd^c \varphi+\omega]^p+\sum_{j=1}^{p-1} \S_j^\omega
(\varphi) \w \eta^{p-j}, \quad \ell\to\infty.   
\end{equation}

More generally, let $\chi_\ell:\R\to \R$ be a sequence of nondecreasing
convex functions, bounded from below, 
that decreases to $t$
as $\ell\to \infty$, and let $\varphi_\ell=\chi_\ell\circ
(\varphi-g)+g$. Then \eqref{krokus} holds for $1\leq p\leq k+1$. 
\end{thm}

 \begin{proof}
It is enough to prove the statement locally. We may therefore assume
that $\varphi=u-h$, where $u$ is psh and $h$ is a smooth
$dd^c$-potential for $\omega$. Let
\begin{equation*}
  u_\ell = \chi_\ell\circ (u-h-g) + h + g.
\end{equation*}
Now Theorem ~\ref{vegetarisk} asserts that
\begin{equation}\label{maxat}
  (dd^c u_\ell)^p\to [dd^cu]^p +
  \sum_{j=1}^{p-1} \S_j(u)\w \big (dd^c (h+g)\big )^{p-j}. 
\end{equation}
Note that $u_\ell=\varphi_\ell + h$. Thus the left hand side of
\eqref{maxat} equals $(dd^c \varphi_\ell + \omega)^p$ and the right
hand side equals $[dd^c\varphi+\omega]^p + \sum_{j=1}^{p-1}
\S_j^\omega(\varphi)\w \eta^{p-j}$. 
   \end{proof}

   \begin{remark}\label{overstyv}
         If $\varphi$ has analytic singularities, then Theorem
     ~\ref{kyrkbacke} 
     follows from
     \cite{B}*{Theorem~1} as in Remark ~\ref{oversattning}. 
     \end{remark}

\section{The Monge-Amp\`ere energy}\label{maenergi}
We want to 
describe $\G_k(X, \omega)$
as finite energy classes. To do this, let us start by recalling the classical
setting. 
If $\varphi\in \psh(X,\omega)\cap L^{\infty}(X)$ then its
Monge-Amp\`ere energy is defined as \eqref{skramla}.  
More generally for $0\leq k\leq n$ and $\varphi\in
\psh(X,\omega)\cap L^{\infty}(X)$ one can define the \emph{Monge-Amp\`ere energy of order $k$} as
\begin{equation*}
    E_k(\varphi):=\frac{1}{k+1}\sum_{j=0}^k\int_X \varphi(dd^c
    \varphi+\omega)^j\wedge \omega^{n-j}.
  \end{equation*}
  These functionals can be extended
  to the entire class $\psh(X,\omega)$ by
  letting
  \begin{equation*}
    E_k(\varphi):=\inf\left \{E_k(\psi): \psi\geq \varphi, \psi\in \psh(X,\omega)\cap
    L^{\infty}(X)\right \},
\end{equation*}
cf.\ \eqref{ramla}. 
We let 
  \begin{equation*}
    \E_k(X,\omega):=\{\varphi\in \psh(X,\omega):
    E_k(\varphi)>-\infty\}
  \end{equation*}
  be the corresponding \emph{finite energy classes}.
 Moreover, we consider the \emph{full mass classes}  
  \begin{equation*}
    \F_k(X,\omega):=\left \{\varphi\in \psh(X,\omega): \int_X
    \langle dd^c \varphi+\omega\rangle^ k\w\omega^{n-k}=\int_X
    \omega^n\right \}.
  \end{equation*}
  Note that
  $$
  \E(X, \omega)=\mathcal E_n(X, \omega) \subset \cdots \subset \E_1(X,\omega), 
  $$
where $\E(X, \omega)$  is the
  standard finite energy class \eqref{famla} corresponding to the energy functional
  $E=E_n$. Similarly,  $\F(X, \omega)=\F_n(X, \omega)$ is
  the standard full mass class \eqref{svamla}.  
%
%
%
%
\begin{prop}\label{massklasserna}
  We have $\F_n(X,\omega)\subset \cdots \subset \F_1(X,\omega).$
  \end{prop}
  \begin{proof}
Assume that $\varphi\in \F_k(X, \omega)$. 
Then
\begin{multline*}
  0=
\int_X \la dd^c\varphi + \omega \ra^k \w
  \omega^{n-k} - \int_X \omega^n 
  =\\
\bigg (\int_X \la dd^c\varphi + \omega \ra^k \w
\omega^{n-k} -
\int_X \la dd^c\varphi + \omega \ra^{k-1} \w
\omega^{n-k+1}\bigg)
+\\
\bigg(\int_X \la dd^c\varphi + \omega \ra^{k-1} \w
\omega^{n-k+1}- 
\int_X \omega^n\bigg ).
 \end{multline*} 
 By Proposition ~\ref{monoton} both terms in
the right hand side are $\leq 0$ and thus they must both vanish. In
particular, $$\int_X \la dd^c\varphi + \omega \ra^{k-1} \w
\omega^{n-k+1}=
\int_X \omega^n$$ and thus $\varphi\in\F_{k-1}(X, \omega)$. 
    \end{proof}

  \begin{remark}\label{relatera}
  Note that $\varphi\in \G_{k-1}(X, \omega)$ is in $\F_k(X, \omega)$ if
  and only if $\S_j^\omega(\varphi)$ vanishes for $j=1,\ldots, k$. 
      \end{remark}

The finite energy classes $\E_k(X, \omega)$ have the following fundamental
properties. 
\begin{thm}\label{grunderna}
  Let $(X, \omega)$ be a compact K\"ahler manifold. Then 
      \begin{enumerate}
      \item
        if $\varphi\in \E_k(X,\omega)$ and $\psi \sim \varphi$
        then $\psi\in \E_k(X,\omega)$;
      \item
        $\E_k(X,\omega)$ is convex;
      \item
        $\E_k(X,\omega)\subseteq \F_k(X,\omega)$.
\end{enumerate}
      \end{thm}

 The last part is a consequence of the second part of Theorem
 ~\ref{grisegenskaper} below, but it also follows from
 \cite{BEGZ}*{Proposition~2.11} (for $k=n$). 
The first two parts follow from the following result.

\begin{prop}\label{ponton}
The functional
$E_k$ is non-decreasing and concave on $\psh(X, \omega)$. 
\end{prop}
It is not hard to see that one can reduce the proof of Proposition
~\ref{ponton} to prove that $E_k$ is non-decreasing and concave on
$\psh(X, \omega)\cap L^\infty(X)$. This, in turn, is an immediate
consequence of the following result.

 \begin{prop}\label{gris1}
If $\varphi$ and $\varphi+u$ are $\omega$-psh and bounded, then 
\begin{equation}\label{energi1}
\frac{d}{dt}\Big|_{t=0^+} E_k(\varphi+tu)=\int_X u (
dd^c\varphi+\omega)^k\w\omega^{n-k} 
\end{equation}
and
\begin{equation}\label{energi2}
\frac{d^2}{dt^2}\Big|_{t=0^+} E_k(\varphi+t u)=-k\int_X du\w d^c u\w (dd^c\varphi+\omega)^{k-1}\w\omega^{n-k}.
\end{equation}
\end{prop}
For $k=n$ this was proved in \cite{BB}*{Propositions~4.1 and~4.4}
and the general case can be
proved by the same arguments.


\section{The non-pluripolar energy}\label{ickepp}

Let us now introduce an alternative way of extending the energies
$E_k(\varphi)$ to the entire class $\psh(X,\omega)$.
\begin{df}\label{samba} 
  Let $(X, \omega)$ be a compact K\"ahler manifold of dimension
  $n$. For  $1\leq k\leq n-1$ we define the
  \emph{non-pluripolar energy of order $k$} of $\varphi\in \psh(X, \omega)$ as  
\begin{equation*}
E_k^{np}(\varphi)=\frac{1}{k+1}\sum_{j=0}^k\int_X \varphi\langle dd^c
\varphi+\omega\rangle^j\wedge \omega^{n-j}.
\end{equation*}
\end{df}
Note that the \emph{non-pluripolar energy} $E^{np}(\varphi)$, as
defined in
Definition ~\ref{mamba}, equals $E^{np}_{n-1}(\varphi)$. 
Note that if $\varphi\in \psh(X, \omega)\cap L^\infty(X)$, then
$E_k^{np}(\varphi)=E_k(\varphi)$.
Moreover, in view of Definition ~\ref{prinsessan}, 
\begin{equation*}
  \G_k(X,\omega)=\{\varphi\in \psh(X,\omega):
  E_k^{np}(\varphi)>-\infty\}.
\end{equation*}

\begin{remark}\label{blind}
Since $0\leq \int_X \langle dd^c
\varphi+\omega\rangle^j\wedge \omega^{n-j}\leq \int_X\omega^n$, if
$C\geq 0$, 
\begin{equation*}
  E^{np}_k(\varphi)\leq E^{np}_k(\varphi+C)\leq
  E^{np}_k(\varphi)+C\int_X\omega^n.
  \end{equation*}
\end{remark}

The functional $E_k^{np}$ is neither monotone nor concave on
$\psh(X,\omega)$ in general, see Examples ~\ref{ejkonvex} and ~\ref{ejmonoton}
below. In particular ${\G}_k(X,\omega)$ is not convex in general. However, we have the following partial generalization of Theorem
~\ref{grunderna}. 

\begin{thm}\label{grisegenskaper}
  Let $(X, \omega)$ be a compact K\"ahler manifold. 
  \begin{enumerate}
  \item\label{argare} Assume that $\varphi, \psi\in \psh(X, \omega)$. If $\varphi\in
    {\G}_k(X,\omega)$ and $\psi\sim \varphi$, then $\psi\in
    {\G}_k(X,\omega)$. 
  \item\label{klassare}
     We have $\mathcal{E}_k(X,\omega)=\mathcal{G}_k(X,\omega)\cap
   \F_k(X,\omega).$
  Moreover, if $\varphi\in \F_k(X, \omega)$, then
  $E_k^{np}(\varphi)=E_k(\varphi)$. 
    \end{enumerate} 
  \end{thm} 
In particular, $\G_k(X, \omega)$ contains the convex subclass $\E_k(X,
\omega)$.
Note that Theorem ~\ref{stranda} follows from Theorem ~\ref{grisegenskaper}.

\smallskip 

The proof relies on the following description of the non-pluripolar energy as a limit of
energies of bounded $\omega$-psh functions and Monge-Amp\`ere masses. 
\begin{prop}\label{gris2}
  If $\varphi\in \psh(X,\omega)$,
then
\begin{equation}\label{gris3}
E_k\big (\max(\varphi,-\ell)\big )+\frac{\ell}{k+1}\sum_{j=0}^k \int_X \big(\omega^j-\la dd^c\varphi + \omega\ra^j\big)\w\omega^{n-j}  \searrow E^{np}_k(\varphi).
\end{equation}
\end{prop}

\begin{proof}
  Let $\varphi_\ell=\max (\varphi, -\ell)$. We claim that
 \begin{multline}\label{lugga}
\int_X\varphi_\ell (dd^c \varphi_\ell + \omega)^j \w\omega^{n-j}+\ell
\int_X \big (\omega^j -\la dd^c \varphi +\omega \ra^j\big ) \w\omega^{n-j}=\\
 \int_X \varphi_\ell  \la dd^c\varphi + \omega\ra^j\w\omega^{n-j}.
\end{multline}
Taking this for granted and noting that the right hand side decreases
to
\newline $\int_X
\varphi\la dd^c\varphi + \omega\ra^j\w\omega^{n-j}$, 
the proposition follows by summing over $j$. 

It remains to prove the claim. 
First, since $\varphi_{\ell}=\varphi$ 
in $O:=\{ \varphi>-\ell\}$, 
which is open in the plurifine topology, since $\varphi_\ell
=-\ell$ in $X\setminus  O$, and since the non-pluripolar
Monge-Amp\`ere operator is local in
the plurifine topology, 
\begin{equation}\label{tugga}
  \varphi_\ell \big ((dd^c \varphi_\ell + \omega)^j  -\la dd^c\varphi
  + \omega\ra^j \big )=
  -\ell \big ((dd^c \varphi_\ell + \omega)^j  -\la dd^c\varphi
  + \omega\ra^j \big ).
\end{equation}
%
Next, since $\varphi_\ell$ is bounded, cf.\ \eqref{ringar}, 
\begin{equation}\label{njugga}
\int_X (dd^c \varphi_\ell + \omega)^j  \w \omega^{n-j}=\int_X \omega^n.
\end{equation}
Combining \eqref{tugga} and \eqref{njugga}, we get \eqref{lugga}. 
\end{proof}

We have the following partial generalization of Proposition
~\ref{ponton}. 
Although $E_k^{np}$ is not monotone on $\psh(X, \omega)$, it is 
monotone on functions of the same singularity type. 

\begin{prop}\label{arma}
  Assume that $\varphi, \psi\in \psh(X, \omega)$. If $\varphi\sim
  \psi$ and 
  $\psi \geq \varphi$, then
  $E_k^{np}(\psi)\geq E_k^{np}(\varphi)$. 
\end{prop}
\begin{proof}
  Since $E_k$ is non-decreasing, see Proposition ~\ref{ponton},
  $E_k(\max(\psi,-\ell))\ge
  E_k(\max(\varphi,-\ell))$. 
  Morover, since $\varphi\sim\psi$, by
  Proposition ~\ref{monoton}, $$\int_X \la dd^c \psi + \omega \ra^j
  \w \omega^{n-j} = \int_X \la dd^c \varphi + \omega \ra^j
  \w \omega^{n-j}.$$
  Now, the proposition follows from Proposition ~\ref{gris2}. 
  \end{proof}

  \begin{proof}[Proof of Theorem ~\ref{grisegenskaper}]
Part \eqref{argare} follows from
    Proposition ~\ref{arma} in view of Remark ~\ref{blind}. 

   It remains to prove part \eqref{klassare}.   Since $E_k(\varphi)=\lim_{\ell\to \infty} E_k(\varphi_\ell)$ and
  $$T:=\sum_{j=0}^k \int_X(\omega^j-\la dd^c \varphi + \omega \ra^j)\w \omega^{n-j}\geq
  0$$ it follows from \eqref{gris3} that $E_k(\varphi)\leq
  E_k^{np}(\varphi)$, and thus $\E_k(X, \omega)\subset \G_k(X,
  \omega)$.  Moreover if $E_k(\varphi)>-\infty$, then $T=0$, since
  clearly 
  $E^{np}(\varphi)$ is bounded from above.  It follows that $\E_k(X,\omega)\subset \F_k(X, \omega)$. 
If $\varphi\in \F_k(X, \omega)$, then $T=0$ by Proposition
~\ref{massklasserna}, and thus $E_k(\varphi)=E_k^{np}(\varphi)$. Hence
$\E_k(X, \omega) = \G_k(X, \omega)\cap \F_k(X, \omega)$.     
\end{proof}

\subsection{Concavity of $E_k^{np}$}\label{derivatorna} 

We have the following generalization of (the second part of)
Proposition ~\ref{ponton}. 

 \begin{prop}\label{konkav}
   Assume that $\varphi\in\G_k(X, \omega)$ has sul. Then
     $E_k^{np}$ is concave on the set of $\psi\in \psh(X, \omega)$ such
   that 
$\psi\sim \varphi$.
\end{prop}
      This is a 
      consequence of the following generalization of \eqref{energi2}. 
  \begin{prop}\label{krikon} 
Assume that $\varphi,\psi\in \G_k(X,\omega)$ have sul and 
$\psi\sim\varphi$. Let $A$ be a closed complete pluripolar set such that
$\varphi$ and thus $\psi$ are locally bounded outside $A$. Moreover,
let $u=\psi-\varphi$. 
Then,
\begin{multline}\label{trex}
  \frac{d^2}{dt^2}\Big|_{t=0^+} E^{np}_k(\varphi+t u)
  =
 - k \int_{X\setminus A} du \w d^c u\w\la dd^c \varphi +\omega
 \ra^{k-1}\w \omega^{n-k}
 \\ 
 - \frac{1}{k+1} 
\sum_{j=2}^k j(j-1)\lim_{\ell\to\infty} \int_{X\setminus (O_\ell\cup A)}
du \w du^c \w ( dd^c\varphi_\ell +\omega )\w \la dd^c \varphi+\omega
\ra^{j-2}\w \omega^{n-j}, 
\end{multline}
where $O_\ell=\{\varphi>-\ell\}\cap\{\psi>-\ell\}$ and $\varphi_\ell=\max
(\varphi, -\ell)$. 
\end{prop}

The right hand side of \eqref{trex} is non-positive, cf.\
Remark ~\ref{popcorn}.
If $\varphi+tu$ is $\omega$-psh also for $t>-\epsilon$ for some
$\epsilon >0$ so that $g(t):=E^{np}_k(\varphi+t u) =E^{np}_k((1-t)\varphi + t
\psi)$ is defined in a neighborhood of $t=0$, then it follows from the proof below that \eqref{trex} is indeed the two-sided second
derivative of $g(t)$ at $t=0$. 
%
It follows that $g(t)$ is concave on the interval $(0,1)$ (or more
generally where it is defined). Thus Proposition ~\ref{konkav}
follows.

For the proof of Proposition ~\ref{krikon} we need the following
lemma, cf.\ Lemma ~\ref{gummistovel}.

\begin{lma}\label{aventyr}
Assume that $\varphi,\psi\in \G_k(X,\omega)$ and $\psi\sim\varphi$.
Then, for any $i\leq j\leq k$, 
\begin{equation*}
\int_X \varphi \la dd^c\varphi+\omega \ra^{i}  \w \la dd^c \psi
+\omega \ra^{j-i}\w\omega^{n-j}>-\infty. 
\end{equation*}
\end{lma}

\begin{proof}
We may assume that
$\varphi, \psi\leq 0$.
Since $\varphi\sim\psi$, it follows from Theorem ~\ref{grisegenskaper}
~\eqref{argare} 
that $\varphi+\psi\in \G_k(X, 2\omega)$.
Thus
\begin{equation}\label{kiting}
\int_X (\varphi + \psi) \big \la dd^c(\varphi + \psi) +2\omega \big
\ra^{j} \w (2\omega)^{n-j}>-\infty
\end{equation}
for $j\leq k$. 
Since the non-pluripolar Monge-Amp\`ere product is multilinear,
\eqref{kiting} is a sum of terms 
\begin{equation}\label{cuisine}
 \int_X \phi \la dd^c\varphi+\omega \ra^{i}  \w \la dd^c \psi
 +\omega \ra^{j-i}\w\omega^{n-j}, 
 \end{equation} 
 where $\phi$ is $\varphi$ or $\psi$.
Since they are all non-positive, the lemma follows. 
\end{proof}

\begin{proof}[Proof of Proposition ~\ref{krikon}]
  Since the non-pluripolar Monge-Amp\`ere product is multilinear,
\begin{equation*}
  E^{np}_k(\varphi+tu)
  =
\frac{1}{k+1}\sum_{j=0}^k\int_X (\varphi+tu) \big \la
 dd^c(\varphi+tu)+\omega\big \ra^j\w\omega^{n-j}  
\end{equation*}
is a polynomial in $t$ with coefficients that are sums of terms of the
form \eqref{cuisine},
where $\phi$ is $\varphi$ or
$\psi$. Since each such integral is finite by Lemma ~\ref{aventyr}, we
may differentiate $E^{np}_k(\varphi+tu)$ formally. Thus 
\begin{multline}\label{prata}
  (k+1)\frac{d^2}{dt^2}\bigg |_{t=0^+} E^{np}_k(\varphi+tu)=
\sum_{j=1}^k2j \int_X u dd^c u \w \big \la dd^c\varphi+\omega\big
\ra^{j-1}\w\omega^{n-j}
+\\ 
\sum_{j=2}^k (j-1)j \int_X  \varphi (dd^c u)^2\w\big \la
dd^c\varphi+\omega\big \ra^{j-2}\w \omega^{n-j}, 
\end{multline}
where 
$dd^c u=\la dd^c\psi+\omega \ra-\la dd^c\varphi+\omega \ra$. 
Since currents of the form $(dd^cu)^\ell \w \la dd^c \varphi + \omega \ra^i$ do not charge
$A$, we may replace $X$ by $X\setminus A$ in \eqref{prata}.

Let $T=  \big \la
dd^c\varphi+\omega\big \ra^{j-2}\w \omega^{n-j}$.
Since $\varphi_\ell:=\max (\varphi, -\ell)$
decreases to $\varphi$, the integral in the $j$th term in the
second sum in \eqref{prata} equals 
\begin{equation}\label{kurva}
  \int_{X\setminus A}  \varphi (dd^c u)^2\w T
  =
  \lim_{\ell\to \infty} \int_{X\setminus A}  \varphi_\ell (dd^c u)^2 \w T.
\end{equation}
Since $dd^cu\w T$ is the difference of two closed positive currents,
we can apply 
Proposition ~\ref{partial} to this. 
It follows that the right hand side of \eqref{kurva} equals 
\begin{equation}\label{andas}
  \lim_{\ell\to \infty} \int_{X\setminus A} u  dd^c u \w (dd^c\varphi_\ell
  +\omega) \w T -
 \int_{X\setminus A} u  dd^c u \w \omega \w T.
\end{equation}
The first term in \eqref{andas} equals
\begin{equation*}
  \lim_{\ell\to \infty} \int_{O_\ell\setminus A} u  dd^c u \w (dd^c\varphi_\ell
  +\omega) \w T + 
 \lim_{\ell\to \infty}\int_{X\setminus (O_\ell\cup A)} u  dd^c u \w (dd^c\varphi_\ell
  +\omega) \w T. 
\end{equation*}
In view of \eqref{husqvarna} 
we conclude that 
\begin{multline*}
\int_{X\setminus A}  \varphi (dd^c u)^2\w\big \la
dd^c\varphi+\omega\big \ra^{j-2}\w \omega^{n-j}
=\\
\int_{X\setminus A}  u dd^c u \w\big \la
dd^c\varphi+\omega\big \ra^{j-1}\w \omega^{n-j}
-
\int_{X\setminus A}  u dd^c u \w\big \la
dd^c\varphi+\omega\big \ra^{j-2}\w \omega^{n-j+1}
+ \\
\lim_{\ell\to\infty} \int_{X\setminus (O_\ell\cup A)}
u dd^c u\w  (dd^c\varphi_\ell +\omega )\w  \la dd^c \varphi+\omega
\ra^{j-2}\w \omega^{n-j}.
  \end{multline*}
Plugging this into \eqref{prata} and, as above, replacing integrals
  over $X$ by integrals over $X\setminus A$, we get
  that
\begin{multline*}
 \frac{d^2}{dt^2}\bigg |_{t=0^+} E^{np}_k(\varphi+tu)=
k \int_{X\setminus A} u dd^c u \w \big \la dd^c\varphi+\omega\big
\ra^{k-1}\w\omega^{n-k}
+\\ 
\frac{1}{k+1}
\sum_{j=0}^k (j-1)j
\lim_{\ell\to\infty} \int_{X\setminus (O_\ell\cup A)}
u dd^c u\w  (dd^c\varphi_\ell +\omega )\w  \la dd^c \varphi+\omega
\ra^{j-2}\w \omega^{n-j}.
\end{multline*}

By Proposition ~\ref{partial}, 
\begin{equation*}
  k \int_{X\setminus A} u dd^c u\w\la dd^c \varphi +\omega \ra^{k-1}\w \omega^{n-k}
  =
  - k \int_{X\setminus A} du \w d^c u\w\la dd^c \varphi +\omega \ra^{k-1}\w \omega^{n-k},
\end{equation*}
which is precisely the first term in the right hand side of
\eqref{trex}.

Next, we claim that
 \begin{multline}\label{godis}
\lim_{\ell\to\infty} \int_{X\setminus (O_\ell \cup A)}
u dd^c u\w  (dd^c\varphi_\ell +\omega )\w  \la dd^c \varphi+\omega
\ra^{j-2}\w \omega^{n-j}
=\\
- \lim_{\ell\to\infty} \int_{X\setminus ( O_\ell\cup A)}
du \w du^c \w ( dd^c\varphi_\ell +\omega )\w \la dd^c \varphi+\omega
\ra^{j-2}\w \omega^{n-j}.
\end{multline}
 Taking this for granted, the last sum in \eqref{andas} equals
 the sum in \eqref{trex}, and thus the proposition follows.
 
 It remains to prove \eqref{godis}. To do this, let $T=\la dd^c
 \varphi+\omega \ra^{j-2}\w \omega^{n-j}$, and write 
 \begin{multline}\label{polkagris}
  \int_{X\setminus (O_\ell\cup A)}
  u dd^c u\w  (dd^c\varphi_\ell +\omega )\w T
  =\\
   \int_{X\setminus A}
u dd^c u\w  (dd^c\varphi_\ell +\omega )\w T
-
\int_{O_\ell\setminus A}
u  dd^c u\w (dd^c\varphi_\ell +\omega )\w T.
\end{multline}
By Proposition ~\ref{partial},  
\begin{equation}\label{morot}
\int_{X\setminus A}
u dd^c u\w (dd^c\varphi_\ell +\omega )\w  T= -\int_{X\setminus A}
du \w d^c u\w  (dd^c\varphi_\ell+\omega) \w T.
\end{equation}
Next, in view of \eqref{husqvarna}, 
\begin{equation}\label{ballong}
 - \lim_{\ell\to \infty} \int_{O_\ell\setminus A}
    u  dd^c u \w ( dd^c\varphi_\ell +\omega ) \w T
 =
  -\int_{X\setminus A}
  u  dd^c u \w \la dd^c\varphi +\omega \ra \w T. 
\end{equation}
 By Proposition ~\ref{partial}, using that the non-pluripolar
 Monge-Amp\`ere operator is local
 in the plurifine topology, this equals
 \begin{multline}\label{talong}
     \int_{X\setminus A}
    du\w d^c u \w \la dd^c\varphi +\omega \ra \w T
    =\\
   \int_{O_\ell\setminus A}
    du \w d^c u \w ( dd^c\varphi_\ell +\omega ) \w T
+ 
    \int_{X\setminus (O_\ell\cup A)}
    du \w d^c u \w \la dd^c\varphi +\omega \ra \w T.
 \end{multline}
Since $u$ is bounded, \eqref{ballong} is finite and thus the second
term in \eqref{talong} tends to $0$ when
$\ell\to \infty$.
We conclude that
\begin{equation}\label{planering}
 -\lim_{\ell\to \infty} \int_{O_\ell\setminus A}
    u  dd^c u \w ( dd^c\varphi_\ell +\omega ) \w T
  =
 \lim_{\ell\to\infty} \int_{O_\ell\setminus A}
    du \w d^c u \w ( dd^c\varphi_\ell +\omega ) \w T
  \end{equation}
Now combining \eqref{polkagris}, \eqref{morot}, and \eqref{planering}
we get \eqref{godis}. 
\end{proof}

\begin{remark}
 By arguments as in the above proof we also get a formula for
 the first derivative of $E(\varphi+t u)$ in the situation of
 Proposition ~\ref{krikon}, cf.\ \eqref{energi1}: 
\begin{multline}\label{kommun}
\frac{d}{dt} E^{np}_k(\varphi+t u)= 
\int_X u\la dd^c(\varphi+tu)+\omega\ra^k\w \omega^{n-k}
+\\
\frac{1}{1+k}\sum_{j=0}^k
j \lim_{\ell\to \infty} \int_{X\setminus O_\ell} u\big
(dd^c(\varphi_\ell+tu_\ell)+\omega) \w \la dd^c(\varphi+tu)+\omega\big
\ra^{j-1}\w\omega^{n-j}.
\end{multline}
In particular,  \eqref{kommun} is non-negative if $u$ is; thus we get an
alternative proof of Proposition ~\ref{arma} in this case. 
\end{remark}

\section{Relative energy}\label{relativa} 
We slightly extend the notion of
relative energy from the introduction. 
 \begin{df}
  Let $\psi\in\G_k(X, \omega)$.
For $\varphi\in \psh(X,\omega)$, such that $\varphi\preceq \psi$ and
for $1\leq k\leq n-1$, we
define the \emph{energy relative to $\psi$ of order $k$} as 
  \begin{equation*}
    E_k^{\psi}(\varphi)=\inf\{E_k^{np}(\varphi'): \varphi'\geq 
    \varphi, \varphi'\sim \psi\}.\end{equation*}
We define the corresponding \emph{finite relative energy classes} 
  \begin{equation*}
  \E_k^\psi(X, \omega) =\{\varphi\preceq \psi,
  E^\psi_k(\varphi)>-\infty\}
  \end{equation*}
  and the \emph{relative full mass classes} 
    \begin{multline*}
      \F_k^{\psi}(X,\omega)=\\\Big \{\varphi\in \psh(X,\omega): \varphi\preceq
      \psi, \sum_{j=0}^k\int_X \langle dd^c \varphi+\omega\rangle^j \wedge
      \omega^{n-j}=\sum_{j=0}^k\int_X \langle dd^c \psi+\omega\rangle^j\wedge
      \omega^{n-j} \Big \}.
      \end{multline*}
\end{df}
  Note that if $\psi=0$, or more generally $\psi\in \psh (X,
  \omega)\cap L^\infty (X)$, then
  $E_k^\psi=E_k$, $\E_k^\psi(X, \omega)=\E_k(X, \omega)$, and 
  $\F_k^\psi(X, \omega)=\F_k(X, \omega)$, cf.\ Proposition
  ~\ref{massklasserna}. 
  Also, note that $E^\psi=E_{n-1}^\psi$, 
  \begin{equation*}
    \E^\psi(X, \omega)=\E_{n-1}^\psi(X,\omega)\subset \cdots \subset
    \E_1^\psi(X,\omega), 
  \end{equation*}
  and
  \begin{equation*}
    \F^\psi(X, \omega) = \F_{n-1}^\psi(X,\omega)\subset \cdots \subset \F_1^\psi(X,\omega),
\end{equation*}
where $E^\psi$, $\E^\psi(X, \omega)$, and $\F^\psi(X, \omega)$ are as
in Definition ~\ref{flamenco}.

We have the following generalization of Theorems ~\ref{grunderna} and
~\ref{grisegenskaper}. 
\begin{thm}\label{psiegenskaper}
  Let $(X, \omega)$ be a compact K\"ahler manifold. Then 
      \begin{enumerate}
      \item\label{knatte}
        if $\varphi\in \E_k^\psi(X,\omega)$ and $\varphi' \sim
        \varphi$, 
        then $\varphi'\in \E_k^\psi(X,\omega)$;
      \item\label{fnatte}
        if $\psi$ has sul, then $\E_k^\psi(X,\omega)$ is convex;
      \item\label{tjatte}
        $\E_k^\psi(X,\omega)=\G_k(X,\omega)\cap \F_k^\psi(X,\omega)$,
        and moreover if $\varphi\in \F_k^{\psi}(X,\omega)$, then
  $E_k^{\psi}(\varphi)=E_k^{np}(\varphi)$.
\end{enumerate}
      \end{thm}
Note that Theorem ~\ref{blanda} corresponds to $k=n-1$. 
For the proof we will use the following
observation. 
\begin{lma}\label{enkelt}
  We have
  \begin{equation*}
    E^\psi_k(\varphi)=\lim_{\ell\to\infty} E^{np}_k\big(\max (\varphi, \psi-\ell)\big).
    \end{equation*}
  \end{lma}
  \begin{proof}
First, note that $\varphi_\ell:= \max (\varphi, \psi-\ell)\sim\psi$ and $\varphi_\ell$ decreases
to $\varphi$. Thus, by Proposition ~\ref{arma}, $\lim_{\ell\to\infty} E_k^{np}(\varphi_{\ell})\geq
E^\psi_k(\varphi)$. 
Next, assume that $\phi^j\sim\psi$ is a sequence decreasing to
$\varphi$. Since $\phi^j\sim \psi$ we can take $\ell_j\to \infty$
such that $\phi^j \geq \psi-\ell_j$ and thus $\phi^j\geq
\varphi_{\ell_j}$. Now $\lim_{j\to\infty} E_k^{np}(\phi^j)\geq
\lim_{j\to\infty} E_k^{np}(\varphi_{\ell_j})=\lim_{\ell\to\infty} E_k^{np}(\varphi_{\ell})$ by Proposition
  ~\ref{arma}. 
    \end{proof}

    We get the following partial 
    generalization of Proposition ~\ref{ponton}. 
    
\begin{prop}\label{armbage}
  Take $\psi\in \G_k(X, \omega)$.
  Then $E^\psi_k$ is non-decreasing.
 If $\psi$ has sul, then  
 $E^\psi_k$ is concave. 
\end{prop}

 \begin{proof}
    As above, let $\varphi_\ell=\max (\varphi, \psi-\ell)$. 
Assume that $\varphi, \varphi'\preceq \psi$ are such that
$\varphi\geq \varphi'$. Then $\varphi_\ell\sim\varphi_\ell'\sim\psi$
and $\varphi_\ell\geq \varphi'_\ell$ and thus, by Proposition
~\ref{arma}, 
$E_k^{np}(\varphi_\ell)\geq E_k^{np}(\varphi'_\ell)$. Taking limits
over $\ell$, in view of Lemma ~\ref{enkelt}, we
get $E_k^\psi(\varphi)\geq E_k^\psi(\varphi')$.

Assume that $\psi$ has sul. It remains to prove that then $E_k^\psi$ is
concave. 
Take $\varphi,
\varphi'\preceq \psi$, $t\in (0,1)$, and a sequence $\phi^j\sim\psi$ decreasing to
$(1-t)\varphi + t\varphi'$. 
Note that we can choose $\ell_j\to \infty$ such that for
each $j$, $\phi^j\geq (1-t) \varphi_{\ell_j} +
t\varphi'_{\ell_j}$. 
Now 
\[
  E_k^{np}(\phi^j) \geq E_k^{np} \big ( (1-t) \varphi_{\ell_j} +
  t\varphi'_{\ell_j} \big )
  \geq
  (1-t) E_k^{np}(\varphi_{\ell_j}) + t E_k^{np}(\varphi'_{\ell_j}).
\]
The first inequality follows since $E_k^{np}$ is nondecreasing on
$\omega$-psh functions of the same singularity type, 
Proposition ~\ref{arma}. The second inequality follows since
$E_k^{np}$ is concave on
$\omega$-psh functions with sul of the same singularity type, Proposition
~\ref{konkav}.
Indeed, note that $\varphi_{\ell}\sim\varphi'_{\ell}\sim (1-t)\varphi_{\ell}+t\varphi'_{\ell}\sim\psi$. 
Now, taking limits over $j$, and $\inf$ over all sequences $\phi^j$, we get
\[
  E_k^\psi((1-t)\varphi + t \varphi') \geq 
  (1-t) E_k^\psi(\varphi) + t E_k^\psi(\varphi').
\]
Thus $E_k^\psi$ is concave. 
    \end{proof}

    \begin{proof}[Proof of Theorem ~\ref{psiegenskaper}]
      Note, in view of Remark ~\ref{blind} and Lemma ~\ref{enkelt}, that $$E^{\psi}_k(\varphi)\leq
      E^{\psi}_k(\varphi+C)\leq E^{\psi}_k(\varphi)+C\int_X\omega^n$$
      if $C\geq 0$. 
      Hence part \eqref{knatte} follows from (the first part of) Proposition ~\ref{armbage}.
Moreover, part \eqref{fnatte} follows immediately from (the second
part of)  Proposition ~\ref{armbage}.

\smallskip 

It remains to prove part \eqref{tjatte}. 
We start by proving
\begin{equation}\label{hallon}
  \G_k(X,\omega)\cap \F_k^{\psi}(X,\omega)\subset
  \E_k^{\psi}(X,\omega).
  \end{equation}
Let 
    \[
      T(\varphi)= \frac{1}{k+1} \sum_{j=0}^k \int_X \big (\omega^j -
      \la dd^c \varphi +\omega \ra^j \big ) \w \omega^{n-j}.
      \]
      Moreover, let $\varphi_\ell=\max (\varphi, \psi-\ell)$. Note
      that $\varphi_\ell\geq \varphi$. 
      Since
      $\varphi_\ell \sim \psi$, it follows from Proposition
      ~\ref{monoton} 
      that
      \begin{equation}\label{jordgubbe} 
        T(\varphi_\ell)=T(\psi).
        \end{equation}
      Now 
      \begin{multline*}
        E_k^\psi(\varphi)=\lim_{\ell\to \infty} E_k^{np}(\varphi_\ell)
        =
      \lim_{\ell\to \infty} \lim_{\lambda\to \infty}
      \big (E_k\big (\max (\varphi_\ell, -\lambda )\big ) + \lambda
      T(\varphi_\ell) \big )
      \geq \\
\lim_{\lambda\to \infty}
      \big (E_k\big (\max (\varphi, -\lambda )\big ) + \lambda
      T(\psi) \big )
      = E_k^{np}(\varphi) + \lim_{\lambda\to \infty}\lambda \big
      (T(\psi)-T(\varphi)\big );
    \end{multline*}
here we have used Lemma ~\ref{enkelt} for the first equality,
Proposition ~\ref{gris2} for the second and last equality, and the
monotonicity of $E_k$ for the inequality. 
Note that $T(\varphi)=T(\psi)$ if and only if $\varphi\in \F^{\psi}_k(X,
\omega)$.
Hence \eqref{hallon} follows. 
Also if $\varphi\in \F^{\psi}_k(X,
\omega)$, then $E^\psi_k(\varphi)=E^{np}_k(\varphi)$. 

\smallskip


To prove the reverse inclusion, consider
\begin{multline}\label{plommon}
  \sum_{j=0}^{k}
  \int_X \varphi_\ell \la dd^c \varphi_\ell\ra^j\w\omega^{n-j}=
  \sum_{j=0}^{k}\int_{O_\ell} \varphi_\ell \la dd^c \varphi \ra^j\w\omega^{n-j}+\\
  \sum_{j=0}^{k} \int_{X\setminus O_\ell} \psi \la dd^c
  \varphi_\ell\ra^j\w\omega^{n-j}
  -
 \ell \sum_{j=0}^{k} \int_{X\setminus O_\ell} \la dd^c
  \varphi_\ell\ra^j\w\omega^{n-j};
\end{multline}
here the equality follows since \eqref{husqvarna} is local in the
plurifine topology. 
By Lemma ~\ref{enkelt}, the left hand side converges to $(k+1)
E^{\psi}_k(\varphi)$. Assume that $\varphi\in \E^\psi_k(X, \omega)$ so
that $E_k^\psi (\varphi)>-\infty$. Note that the three terms in the right
hand side are bounded from above. Thus each of them is $>-\infty$. 
%
%
In particular,  
\[
  \sum_{j=0}^{k}\int_{O_\ell} \varphi_\ell \la dd^c \varphi
  \ra^j\w\omega^{n-j}
  \to (k+1) E^{np}_k(\varphi)>-\infty, 
\]
which means that $\varphi\in \G_k(X, \omega)$. 
Moreover,
the finiteness of the third term in the right hand side
of \eqref{plommon} implies that
\begin{multline*}
  0=
 \lim_{\ell\to \infty} \sum_{j=0}^{k} \int_{X\setminus O_\ell} \la dd^c
  \varphi_\ell\ra^j\w\omega^{n-j}=\\
  \lim_{\ell\to\infty}\sum_{j=0}^{k} \int_{X} \la dd^c
  \varphi_\ell\ra^j\w\omega^{n-j}- \lim_{\ell\to \infty} \sum_{j=0}^{k} \int_{O_\ell} \la dd^c
  \varphi_\ell\ra^j\w\omega^{n-j}
  = \\
   \sum_{j=0}^{k} \int_{X} \la dd^c
  \psi \ra^j\w\omega^{n-j}- \sum_{j=0}^{k} \int_{X} \la dd^c
  \varphi \ra^j\w\omega^{n-j}, 
  \end{multline*}
  i.e., $\varphi\in \F^{\psi}_k(X, \omega)$. Here we have used
  \eqref{jordgubbe} for the last equality. 
 \end{proof}

It is quite possible that the more general integration by parts
results \cite{Xia19}*{Theorem~1.1} and \cite{Duc20}*{Theorem~2.6},
cf.\ Remark \ref{objektiv}, can be used remove the sul assumption from Thm~\ref{psiegenskaper} \eqref{fnatte}. In fact Vu has already proved the convexity of certain related finite relative energy classes \cite{Duc20}*{Theorem~1.1}.

\section{Relations to the B\l{}ocki-Cegrell class}\label{dumle} 

In \cite{B06} the domain of the Monge-Amp\`ere operator was characterized in
several equivalent ways. In order to prove Theorem ~\ref{globalbc} we will use the
following characterization from \cite{B06}*{Theorem~1.1}. 
\begin{prop}\label{humle} 
Let $\Omega$ be an open subset of $\C^n$. Then $u\in \psh(\Omega)$ is
in $\D(\Omega)$ if and only if for each open $\U\subset \Omega$ and
any 
sequence of smooth $u_\lambda\in \psh (\U)$ decreasing to $u$ in $\U$, the
sequences 
\begin{equation*}
  |u_\lambda|^{n-j-2}du_\lambda \w d^cu_\lambda\w (dd^c u_\lambda)^j
  \w \omega^{n-j-1}, ~~~~~ j=0,1,\ldots, n-2,
\end{equation*}
where $\omega$ is a smooth strictly positive $(1,1)$-form, are locally
weakly bounded in $\U$. 
  \end{prop}


We will also use the following lemma. 
  
\begin{lma}\label{stolsben} 
Assume that $\{\U_i\}$ is a finite open covering
of $X$ and that, for each $i$, 
$g_i$ is a local $dd^c$-potential of $\omega$ in $\U_i$. Moreover,
assume that $\chi_i$ is partition of unity subordinate $\{\U_i\}$.
Then there
is a $C>0$, such that, if 
$\varphi\in \psh(X, \omega)$ is smooth and $\varphi<0$, then, for $1\le
j\le n-1$, 
\begin{multline}\label{merkurius}
  \int_X -\varphi (dd^c\varphi+\omega)^{j}\w\omega^{n-j}\\
    \leq 2 \sum_i \int_X \chi_i d(\varphi+g_i)\w d^c (\varphi +g_i) \w
  (dd^c\varphi +\omega)^{j-1}\w\omega^{n-j} +\\
  \int_X -\varphi (dd^c\varphi+\omega)^{j-1}\w\omega^{n-j+1} + C.
\end{multline}
\end{lma}

\begin{proof}
We will use the following statement. 
Assume that $A, B, a, b$ are real smooth functions. Then 
\begin{equation}\label{solsken}
  A B da\w d^c b+   A B db\w d^c a = 2AB (da\w d^c b)_{(1,1)} 
  \leq  
A^2 da\w d^ca + B^2 db\w d^c b
\end{equation}
as forms, where $(1,1)$ denotes the component of bidegree $(1,1)$.  This is an immediate consequence of the 
inequality
$ i\alpha\w\bar\beta+i\beta\w\bar\alpha 
  \leq i\alpha\w\bar\alpha +  i\beta\w\bar\beta$ 
for $(1,0)$-forms $\alpha, \beta$, applied to $\alpha=A\partial a$ and
$\beta=B\partial b$.

\smallskip 
  
  Take $1\leq j \leq n-1$, let $T=(dd^c\varphi+\omega)^{j-1}\w\omega^{n-j}$, and let
  \[
    I=\sum_i \int_X \chi_i d(\varphi+g_i)\w d^c (\varphi +g_i) \w T. 
      \]
 Then, by Stokes' theorem, the left hand side of \eqref{merkurius} equals
  \begin{equation}\label{pluto}
    \int_X -\varphi (dd^c \varphi + \omega) \w  T  = 
    \int_X d\varphi \w d^c \varphi \w  T + \int_X -\varphi \w\omega T.
  \end{equation}
Moreover, 
  \begin{multline}\label{uranus} 
  \int_X d\varphi \w d^c \varphi \w  T=  \sum_i \int_X \chi_i d\varphi
  \w d^c \varphi \w  T
  =\\
    I - \sum_i \int_X \chi_i d\varphi \w d^c g_i \w  T -
    \sum_i \int_X \chi_i d g_i \w d^c \varphi \w  T - \sum_i \int_X \chi_i d g_i \w d^c g_i \w  T 
    \\
    \leq
   I + \frac{1}{2} \sum_i \int_X \chi_i d\varphi \w d^c \varphi \w  T  + 
    \sum_i \int_X \chi_i d g_i \w d^c g_i \w  T = 
   \\
    I + \frac{1}{2}  \int_X d\varphi \w d^c \varphi \w  T + 
    \sum_i \int_X \chi_i d g_i \w d^c g_i \w  T 
 \end{multline}
 Here, the inequality follows by \eqref{solsken} applied to
 $A=-1/\sqrt 2$, $B=\sqrt 2$, $a=\varphi$, and $b=g_i$. 
  Note that there is a $D>0$, such that for all $i$, 
 $
   dg_i\w d^c g_i\w T\leq D \omega \w T
 $
 as forms. Thus, from \eqref{uranus} we conclude that 
 \begin{equation}\label{tellus}
   \frac{1}{2}\int_X d\varphi \w d^c \varphi \w  T
   \leq I + D  \int_X \omega\w T = I+ D  \int_X \omega^n. 
 \end{equation}
 Combining \eqref{pluto} and \eqref{tellus}, we get
 \eqref{merkurius} (with $C=2 D \int_X \omega^n$). 
 \end{proof} 

 \begin{proof}[Proof of Theorem ~\ref{globalbc}]
   Assume that $\varphi\in \D(X, \omega)$. Clearly, we may assume that
   $\varphi<0$. 
Let $\varphi_\ell=\max
(\varphi, -\ell)$. Moreover choose a sequence $\epsilon_\lambda\to 0$,
$\lambda\in \N$. By  \cite{Dem93}*{Theorem~1.1}, for each $\ell$, there is a sequence 
$\varphi_{\ell, \lambda}$ 
of smooth negative $(1+\epsilon_{\lambda})\omega$-psh functions decreasing to
$\varphi_\ell$. 
We may assume that $\varphi_{\ell,\lambda}>\varphi$; otherwise replace
$\varphi_{\ell, \lambda}$ by $\varphi_{\ell, \lambda}+\delta_{\ell,
  \lambda}$, where $\delta_{\ell, \lambda}\to 0$. 
Since $X$ is compact we may inductively choose $\lambda_\ell\geq \ell$ so that
$\varphi_{\ell, \lambda_\ell}<\varphi_{\kappa, \lambda_{\ell-1}}$ for
all 
$\kappa<\ell$, and 
\begin{equation}\label{veckan}
 \int_{X} -\varphi_{\ell} (dd^c \varphi_{\ell}+\omega)^ j\w
 \omega^{n-j} < 
  \int_{X} -\varphi_{\ell, \lambda_\ell} (dd^c \varphi_{\ell, \lambda_\ell}+\omega)^ j\w
  \omega^{n-j} + 1/\ell. 
\end{equation}
Set $\psi_\ell=\varphi_{\ell, \lambda_\ell}$.
Then $\psi_\ell$ is a sequence of smooth negative
$(1+\epsilon_{\lambda_\ell})$-psh functions decreasing to
$\varphi$. 
Assume that $\{\U_i\}$, $g_i$, and $\chi_i$ are as in Lemma
\ref{stolsben}. 
Then $\varphi+g_i\in \D(\U_i)$ and
$\psi_\ell+(1+\epsilon_{\lambda_\ell})g_i$ is a sequence of smooth psh functions
decreasing to $\varphi+g_i$. It follows by
Proposition ~\ref{humle} that there is a $K>0$, such that, for each
$\ell$ and $j=0, \ldots, n-2$, 
\begin{equation*}
   \sum_i \int_X \chi_i d\big (\psi_{\ell}+(1+\epsilon_{\lambda_\ell})
   g_i \big)\w d^c \big(\psi_{\ell} +(1+\epsilon_{\lambda_\ell}) g_i \big ) \w
  \big (dd^c\psi_{\ell} +(1+\epsilon_{\lambda_\ell})\omega\big
  )^{j}\w \big ((1+\epsilon_{\lambda_\ell})\omega\big )^{n-j-1}\leq K.
\end{equation*}
By Lemma ~\ref{stolsben}, inductively applied to $j=1, \ldots,
n-1$, we get that there is an $M>0$ such that, for $1\leq j\leq
  n-1$, 
\[
 \int_X -\psi_{\ell} (dd^c \psi_{\ell}+\omega)^ j\w \omega^{n-j} \leq M.
\]
By \eqref{veckan}, 
\[
  \int_{O_\ell} -\varphi_{\ell} (dd^c \varphi_{\ell}+\omega)^ j\w
  \omega^{n-j} \leq M +1/\ell.
\]
Thus, in view of \eqref{husqvarna}, as desired, 
\[
  \int_X -\varphi \la dd^c \varphi +\omega\ra ^ j\w
  \omega^{n-j} \leq M. 
\]
   \end{proof}

\section{Examples of functions with finite non-pluripolar
  energy}\label{exempelsamling}

Let us present some exemples of functions with finite non-pluripolar
energy. 
As in
the local situation $\omega$-psh functions with analytic singularities
are in $\G_k(X,\omega)$.

\begin{ex}\label{salta} 
Assume that $\varphi\in \psh(X, \omega)$ has analytic
singularities, i.e., locally $\varphi=c\log |f|^2+b$, where $c>0$, $f$ is a
tuple of holomorphic functions, and $b$ is bounded. Then $\varphi\in
\G(X, \omega)$ by Example ~\ref{aladab}.
Note that
$\varphi\in \F(X, \omega)$ if and
only if $\varphi$ is locally bounded, cf.\ Remark ~\ref{relatera}. 

Moreover, note that if $\varphi, \psi\in \psh(X, \omega)$ have
analytic singularities, then $(1-t)\varphi+t\psi\in \psh(X, \omega)$
has analytic singularities. Thus the set of functions in $\G_k(X,
\omega)$ with analytic singularities is a convex subclass.
  \end{ex}

  We will consider some examples on $\P^n=\P^n_x$ with homogeneous coordinates $[x]=[x_0,\ldots,
x_n]$ and equipped with the Fubini-Study form
\begin{equation}\label{nassla}
  \omega_{\text{FS}}=dd^c\log |x|^2=\log (|x_0|^2+\cdots +
  |x_n|^2).
  \end{equation}

\begin{ex}\label{dansa} 
Let $f$ be a $\lambda$-homogeneous polynomial on $\C^{n+1}$ that we
consider as a holomorphic section of $\Ok(\lambda)\to \P^n$.
Now 
$\varphi:=\log (|f|^2/|x|^{2})=\log|f|^2_{\text{FS}}$, where $|\cdot|_{\text{FS}}$ denotes
the Fubini-Study metric, is in $\psh(\P^n, \lambda
\omega_{\text{FS}})$ and has analytic singularities. In particular,   
$\varphi\in \G(\P^n, \lambda\omega_{\text{FS}})$ by Example ~\ref{salta}.
\end{ex}

Next, let us consider examples of functions with finite non-pluripolar
energy 
that do not have
analytic singularities.
First, we present a global version of Example ~\ref{utvidgat} that shows
that $\G_k(X, \omega)$ is not convex in
general. For this we need
the following lemma.
\begin{lma}\label{lillemma}
  Assume that $\varphi\in\psh(X, \omega)$, where $(X, \omega)$ is a
  compact K\"ahler manifold. Moreover, assume that $g:\R\to\R$ is smooth, 
  non-decreasing and convex, and that $g'(\varphi)\leq 1$. Then
  $g(\varphi)\in \psh(X, \omega)$. 
  \end{lma}
  \begin{proof}
    Note that
    \[
      dd^c \big (g(\varphi) \big )+\omega=
      g''(\varphi) d\varphi\w d^c \varphi + g'(\varphi) (dd^c
      \varphi+\omega) + \big (1-g'(\varphi)\big )\omega, 
    \]
    and that each term in the right hand side is $\geq 0$ 
    by the
    assumptions on $\varphi$ and ~$g$. 
  \end{proof}

\begin{ex}\label{ejkonvex}
Let $X$ be a projective manifold and $L\to X$ an ample line
bundle equipped with a positive metric $h$ with corresponding K\"ahler
form $\omega$.
Moreover, let $f=(f_1,\ldots, f_m)$ be a holomorphic section of
$L^{\oplus m}$, take $C>0$ such that $|f|^2_h =|f_1|_h^2+\cdots +
|f_m|_h^2< C$, and let $\varphi=\log|f|^2_h-C$. 
For $\epsilon \in (0,1)$, let $g(t)=-(-t)^\epsilon$ and
\[
  \psi=g(\varphi)=-(-\log|f|^2_h+C)^\epsilon.
\]
Then $\psi\in \psh(X, \omega)$ by Lemma ~\ref{lillemma}.
As in
Example ~\ref{utvidgat} one sees that, for each $k$, $\psi\in \G_k(X, \omega)$ if and
only if $\epsilon <1/2$, and that $(1-t)\varphi+t\psi
\in \psh(X,
\omega)\setminus \G_k(X, \omega)$
for all $t\in (0,1)$ and all
$\epsilon \in (0,1)$.
\end{ex}

The previous example also shows that $E^{np}_k$ is not monotone in
general.
  \begin{ex}\label{ejmonoton}
Let us use the notation from Example ~\ref{ejkonvex}. Note that
$g(t)>t$ for $t\leq -1$. It follows that $\psi >\varphi$, and thus
$\varphi':= (1-t)\varphi+t\psi> \varphi$ if $t\in (0,1)$. By 
Example ~\ref{ejkonvex}, $E^{np}_1(\varphi')=-\infty$, whereas
$E^{np}_1(\varphi)>-\infty$. Thus $E^{np}_1$ is not non-decreasing. 
    \end{ex}

\smallskip 

Next, let us consider some examples of functions with finite
non-pluripolar energy that do not have sul. 
The following is a global version of Example ~\ref{plutt}. 
\begin{ex}\label{skrutt}
  Let $(X, \omega)$ be a compact K\"ahler manifold of dimension $n$.
  For $i=1,2,\ldots$, let $\psi_i$ be an $\omega$-psh function with analytic
singularities and let
$b_i>0$.
If $B:=\sum b_i<\infty$, then 
\begin{equation*}
\varphi :=\sum_1^\infty b_i \psi_i \in \psh (X, B\omega).
\end{equation*}
Given $C>0$, by the same arguments as in Example ~\ref{plutt} we can
choose $b_i$ inductively so that 
\begin{equation*}
\int_{X} \varphi \la dd^c
\varphi+B\omega\ra^j\w(B\omega)^{n-j}\geq -C , \quad j\le n-1, 
\end{equation*}
and thus $\varphi\in \G(X, B\omega)$. 
\end{ex}

The following global variant of Example ~\ref{punkt}  gives
an example of a function with finite non-pluripolar energy that
neither has sul nor full mass.
\begin{ex}\label{kplan}
 Let $\sigma=(\sigma_1,\ldots, \sigma_k)$ be a holomorphic section of
$\O(1)^{\oplus k}\to \P^n_x$, 
such that the zero set is a codimension $k$-plane $P_\sigma$, and let
$\psi_\sigma=\log (|\sigma|^2/|x|^2)$. 
Then $\psi_\sigma\in \psh(\P^n, \omega_{\text{FS}})$ has anaytic
singularities and unbounded locus  $P_\sigma$. Next, let
$\sigma^1,\sigma^2, \ldots$ be holomorphic sections of $\O(1)^{\oplus
  k}$ that
define codimension $k$-planes $P_{\sigma^1}, P_{\sigma^2}, \ldots$,
respectively, 
such that $\bigcup P_{\sigma^i}$ is dense in $\P^n$. 
Let 
  $\psi_i=\psi_{\sigma^i}$, and let 
  $\psi=\sum b_i \psi_i$ be constructed as in Example
  ~\ref{skrutt}.
Then $\varphi:=\psi/B \in \G(\P^n, \omega_{\text{FS}})$, where
$B=\sum_i b_i$, but $\varphi$ is
  not locally bounded anywhere; in particular, $\varphi$
  does not have sul.
Moreover, since $\psi_i\notin \F_k(\P^n, \omega_{\text{FS}})$, cf.\ Example
~\ref{salta}, it follows that $\varphi\notin \F_k(\P^n, \omega_{\text{FS}})$ and thus $\varphi\notin\E_k(\P^n,
\omega_{\text{FS}})$.
\end{ex}

Next, let us consider a variant of Example ~\ref{kplan} in $\P^1_x$. 
\begin{ex}\label{kyckling} 
  Let $\psi_i=\max (\log|x_1-a_ix_0|_{\text{FS}}^2, -c_i)$, where $(1, a_i)\in
  \P^1$ and $c_i\in (0,\infty]$. Then $\psi_i\in \psh(\P^1, \omega_{\text{FS}})$ has analytic
  singularities and as in Example ~\ref{skrutt} we can choose $b_i$ so
  that
  \[
    \varphi=\sum_{i=1}^\infty b_i \psi_i =
    \sum_{i=1}^\infty b_i \max (\log|x_1-a_ix_0|_{\text{FS}}^2, -c_i)
  \]
  is in $\G(\P^n,
  B\omega_{\text{FS}})$, where $B=\sum_i b_i$. 
 Note that if $c_i=\infty$ we are in the situation in Example ~\ref{kplan}.
 
 Moreover, if the points $(1,a_i)$ are dense in $\P^1$ and $b_ic_i\to
 \infty$, then $\varphi$ is not locally bounded anywhere; in
 particular, it does not have sul. 
  By choosing $b_i$ and $c_i$ inductively it is 
  possible to arrange this so that in addition $\varphi\in \G(\P^n,
  B\omega_{\text{FS}})$. One can check that if $b_i\to 0$ fast enough
  and $b_i^2c_i\to 0$, then, in fact, $\varphi\in \E(\P^n,
  B\omega_{\text{FS}})$.
\end{ex}

\subsection{Product spaces}\label{prodsection}
Let $(X_1, \omega_1)$ and $(X_2, \omega_2)$ be compact K\"ahler manifolds. Let
$X=X_1\times X_2$ and $\omega=\pi^*_1\omega+\pi_2^*\omega_2$, where 
$\pi_i\colon X\to X_i$ are the natural projections. 
Then it is readily verified
that $(X, \omega)$ is a
compact K\"ahler manifold. Moreover, if $\varphi^i\in \psh(X_i, \omega_i)$,
then $\varphi:=\pi^*_1\varphi^1+\pi^*_2\varphi^2\in \psh(X, \omega)$. 
We have the following global version of Proposition ~\ref{pengar} that
can be proved in the same way. 
\begin{prop}\label{prodspace}
Let $(X_i,\omega_i)$, $(X, \omega)$, $\varphi^i$, and $\varphi$ be as
above. Assume that $\varphi^i\in \G_k(X_i,\omega_i)$, $i=1,2$. Then $\varphi\in G_k(X,\omega)$. 
\end{prop}

We can use Proposition ~\ref{prodspace} to find new non-trivial
examples of $\omega$-psh functions with finite non-pluripolar
energy.

\begin{ex}\label{konkretex}
  Take $\varphi^1\in \psh(\P^n, \omega_{\text{FS}})$ that has analytic
  singularities and is not locally bounded. Then $\varphi^1\in \G(\P^n, \omega_{\text{FS}})$ but $\varphi^1\notin \F(\P^n, \omega_{\text{FS}})$. 
Next, take $\varphi^2\in \E(\P^1, \omega_{\text{FS}})$ that does not have sul,
e.g., as in Example ~\ref{kyckling}.
Then, by Proposition
~\ref{prodspace}, 
$\varphi=\pi^*_1\varphi^1+ \pi_2^* \varphi^2 \in \G_k(\P^n\times \P^1,
\pi^*_1\omega_{\text{FS}}+\pi^*_2 \omega_{\text{FS}})$, but $\varphi$
neither has sul nor full mass.
\end{ex}

\subsection{Convex combinations in projective space}\label{joina}

Througout this subsection we assume that $\P^n=\P^n_x$ is equipped with
the Fubini-Study form \eqref{nassla} that we denote by $\omega$ or
$\omega_{\P^n}$. 
We know from Example ~\ref{ejkonvex} that convex combinations of functions
in $\G_k(\P^n,\omega)$ are not in $\G_k(\P^n,\omega)$ in general. 
It turns out, however, that if $\varphi^i\in \G_k(\P^{n_i}, \omega)$
for $i=1,2$, then there are natural associated functions $\tilde\varphi^i\in \G_k(\P^{N}, \omega)$,
where $N=n_1+n_2+1$, such that any convex combination of $\tilde
\varphi_1$ and $\tilde \varphi_2$ is in 
$\G_k(\P^{N}, \omega)$.

To make this more precise, let us first settle some
notation. 
Recall that $\P^n$ is covered by coordinate charts
$\U_i=\U_i^{\P^n}=\{x_i\neq 0\}$. In $\U_0=\{x_0\neq
0\}=\{[1,x_1,\ldots x_n]\}$ we have local coordinates
$x'=(x_1,\ldots, x_n)$ and  
$\gamma_{x'}:=dd^c \log (1+|x'|^2)$ is a local potential for $\omega$.

Next, note that the maps
$p_1: \P^N_{x,y}\dashrightarrow \P^{n_1}_x, [x,y]\mapsto [x]$ and
$p_2: \P^N_{x,y}\dashrightarrow \P^{n_2}_x, [x,y]\mapsto [y]$ are
well-defined outside the planes $\{x=0\}$ and $\{y=0\}$, respectively. If
we let $\pi: Y\to \P^N$ be the blowup of $\P^N$ along these planes,
then there are maps $\hat p_i: Y\to \P^{n_i}$, $i=1,2$, so that the diagrams 
\begin{equation*}
\xymatrix{
Y \ar[d]_{\pi} \ar[dr]^{\hat p_i} & \\
\P^N \ar@{-->}[r]_{{p_i}} & \P^{n_i}
}
\end{equation*}
commute. 
Moreover, if
$\varphi\in \psh (\P^{n_i}, \omega)$, then $p_i^*\varphi:=\pi_*\hat
p_i^* \varphi$ is an upper semicontinuous function on $\P^N$ with
possible singularities along $\{x=0\}$ or $\{y=0\}$. If we understand
$p_i^*\varphi$ as the upper semicontinuous regularization we get a 
well-defined upper semicontinous function on
$\P^N$.

Now, assume that $\varphi\in \psh(\P^{n_1}, \omega)$. We want to show
that there is a natural associated 
$\tilde \varphi \in \psh (\P^N, \omega)$. 
 Consider the functions $\Gamma_x =\log |x|^2$ and $\Gamma_{x,y}=\log
 |x,y|^2$ on
 $\C^{N+1}$. Note that $\Gamma_x-\Gamma_{x,y}$ is a $0$-homogenous function on
 $\C^{N+1}$; thus it defines a global function on $\P^N$.
 Now, let
 \begin{equation*}
   \tilde \varphi = p_1^* \varphi + \Gamma_x-\Gamma_{x,y}.
   \end{equation*}
We claim that $\tilde \varphi\in \psh(\P^N, \omega)$. To prove this we
need to show 
that $dd^c \tilde \varphi + \omega_{\P^N}\geq 0$. In fact, it suffices to do
this outside $\{x=0\}$, since $\{x=0\}$ has codimension at least $2$
and $dd^c \tilde \varphi$ is a normal $(1,1)$-current and thus we may assume that $p_1$ is holomorphic. Note that
$dd^c (\Gamma_x-\Gamma) = p_1^* \omega_{\P^{n_1}} -\omega_{\P^N}$. Thus $dd^c \tilde \varphi + \omega_{\P^N}= p_1^* (dd^c\varphi +\omega_{\P^{n_1}})\geq
0$, since $\varphi$ is $\omega_{\P^{n_1}}$-psh. 

Analogously, if $\varphi\in \psh(\P^{n_2}, \omega)$,
$\tilde \varphi:= p_2^* \varphi+\Gamma_y-\Gamma_{x,y}$, where
$\Gamma_y=\log |y|^2$, is a well-defined function in $\G_k(\P^{N}, \omega)$.

\begin{prop}\label{joinprop}
Assume that $\varphi^i\in \G_k(\P^{n_i}, \omega), i=1,2$. Then for
$0\leq t\leq 1$, $(1-t)\tilde\varphi^1+t\tilde\varphi^2\in
\G_k(\P^{n_1+n_2+1}, \omega)$, where $\tilde\varphi^i$ are defined as
above. 
\end{prop}

\begin{proof}
  Let $\varphi=(1-t)\tilde\varphi^1+t\tilde \varphi^2$ and, as above, let $N=n_1+n_2+1$.
  We need to prove that for $0\leq j\leq k$, 
  \begin{equation}\label{varvinter}
\int_{\U} \varphi\la dd^c\varphi+\omega\ra^j \w\omega^{N-j} >-\infty, 
\end{equation}
where $\U\subset \P^N$ is open.
We may assume that $\U$ has compact support in one of the coordinate
charts $\U_i^{\P^N}$, say in $\U_0^{\P^N}$ with homogeneous coordinates
$[1,x',\lambda, \lambda y']$. 

Let us consider the integrand in \eqref{varvinter} in these
coordinates.
First note that
$$\Gamma_x=\log (1+|x'|)=\gamma_{x'} \quad \text{ and } \quad 
\Gamma_y=\log \big (|\lambda|^2(1+|y'|)\big )=\log |\lambda|^2 +
\gamma_{y'}.$$
We write $\gamma$  for $\Gamma_{x,y}$.
Since $\varphi^i$ are $0$-homogeneous, 
$p_1^*\varphi^1$ and $p_2^*\varphi^2$ only depend on $x'$ and $y'$,
respectively; 
we write
$p_1^*\varphi^1=\varphi^1(x')$ and $p_2^*\varphi^2=\varphi^2(y')$. 

We start with the factor 
$\la dd^c\varphi+\omega\ra^j = \la dd^c (\varphi + \gamma) \ra^j$. 
Now
\begin{equation*}
  \varphi+\gamma=(1-t)\tilde\varphi^1+t\tilde \varphi^2+\gamma
  =(1-t)\big(\varphi^1(x')+\gamma_{x'}\big)
  +t\big(\varphi^2(y')+\log |\lambda|^2+\gamma_{y'}\big ). 
\end{equation*}
Note
that $dd^c \log|\lambda|^2=[\lambda=0]$; in particular, it has support
where $\log |\lambda|^2=-\infty$. It follows that
\begin{equation*}
  \la dd^c\varphi+\omega\ra^j
  =
  \big \la (1-t) (dd^c \varphi^1(x')+ \omega_x)+
  t(dd^c \varphi^2 (y')+ \omega_y)\big \ra^j, 
  \end{equation*}
where $\omega_x=dd^c\gamma_{x'}$ and $\omega_y=dd^c \gamma_{y'}$.

  Next, let $\theta=\omega_x+\omega_y+\eta_\lambda$, where
  $\eta_\lambda =id\lambda \w d\bar\lambda$. Then $\theta$ is a
  smooth strictly  positive $(1,1)$-form on $\P^N$. Thus we can replace $\omega^{N-j}$ in
  \eqref{varvinter} by $\theta^{N-j}$. Note that for degree reasons
  \eqref{varvinter} is a sum of terms of the form
  \begin{equation}\label{reklam}
\int_{\U} \varphi \big \la (1-t) (dd^c \varphi^1(x')+ \omega_x)+
  t(dd^c \varphi^2 (y')+ \omega_y)\big \ra^j (\omega_x +
  \omega_y)^{n_1+n_2-j}\w \eta_\lambda.
    \end{equation}

    Finally consider
    \begin{equation*}
      \varphi=(1-t)\tilde \varphi^1+t \tilde \varphi^2 =
        (1-t)\varphi^1(x')+t \varphi^2(y') + t\log|\lambda|^2 + (1-t)\gamma_x+t\gamma_y-\gamma.
     \end{equation*}
Since the last three terms are smooth, 
the contribution from these terms in \eqref{reklam} is finite. 
Next, the contribution from $\log |\lambda|^2$ is finite since
$\log|\lambda|^2$ is (locally) integrable. 
Finally, if we replace $\varphi$ by $\varphi^1(x')$ or
$\varphi^2(y')$, then we are in the situation in Section
~\ref{prodsection}. 
Indeed, we can regard $\U$ as a relatively compact subset of
$\U_0^{\P^{n_1}}\times
\U_0^{\P^{n_2}}\times \C_\lambda\subset
\P^{n_1}\times
\P^{n_2}\times \C_\lambda$. 
Then $\varphi^1(x')$, $\omega_x$, $\varphi^2(y')$, and $\omega_y$ are
just the pullbacks of $\varphi^1$, $\omega_{\P^{n_1}}$, $\varphi^2$,
and $\omega_{\P^{n_2}}$, respectively, under the natural projections $\U\to
\U_0^{\P^{n_i}}$. Thus by Proposition ~\ref{prodspace} this 
contribution is also finite. We conclude that
\eqref{varvinter} holds and hence $\varphi\in \G_k(\P^N, \omega)$. 
\end{proof}



\begin{ex}\label{kaffekopp}
  Let $\varphi^i$, $i=1,2$ be as in Example ~\ref{konkretex}. 
Then   
 $\varphi:= (1-t)\tilde\varphi^1+t\tilde\varphi^2\in
\G_k(\P^{n+1}, \omega)$ and $\varphi$ neither has analytic singularities nor
full mass. 
\end{ex}

\begin{remark}
Let us consider the complex Monge-Amp\`ere equation
  \begin{equation}\label{calabi}
    \la dd^c \varphi + \omega \ra^n = f\omega^n.
  \end{equation}
  on the compact K\"ahler manifold $(X, \omega)$. 
  By Yau's solution to the famous Calabi conjecture, \eqref{calabi} has a smooth
  solution if $f$ is smooth.
   This result has been generalized
  by several authors in many different directions allowing less
  regular $f$.
   For instance, in \cite{DDL}*{Theorem~1.4} it was proved that under certain
  conditions on $f$ there is a unique solution $\varphi\in \psh(X,\omega)$ with
  prescribed singularity type. 
  
  One could hope that it would be possible to use our Monge-Amp\`ere currents to
  solve a Monge-Amp\`ere equation of the form
       \begin{equation}\label{babbla}
    [dd^c \varphi + \omega ]^n = \mu
  \end{equation}
where $\mu$ is allowed to be a more general current. For instance, assume that 
$\mu=\mu_1+\mu_2$, where $\mu_1=f\omega^n$ and $\mu_2$ has support on
a subvariety $Z\subset X$, defined by a holomorphic section $s$ of
a Hermitian vector bundle over $X$. Then it might be natural to
look for a solution $\varphi\sim\log|s|^2$ to \eqref{babbla}. Note
that if $\varphi$ solves \eqref{babbla}, then, in particular, $\la dd^c
\varphi + \omega \ra^n = \mu_1$. Now, by \cite{DDL}*{Theorem~1.4} this
completely determines $\varphi$. Thus one can only hope to solve
\eqref{babbla} for very special choices of $\mu_1$ and $\mu_2$. 
  \end{remark}


\begin{bibdiv}
\begin{biblist}


  \bib{A}{article}{
    AUTHOR = {Andersson, Mats},
     TITLE = {Residues of holomorphic sections and {L}elong currents},
   JOURNAL = {Ark. Mat.},
     VOLUME = {43},
      YEAR = {2005},
    NUMBER = {2},
     PAGES = {201--219},
      ISSN = {0004-2080},
}
  

\bib{ABW}{article}{
   author={{Andersson}, Mats},
   author={ {B{\l}ocki}, Zbigniew},
   author={Wulcan, Elizabeth},
   title={On a Monge-Amp\`ere operator for plurisubharmonic functions with analytic singularities},
   JOURNAL = {Indiana Univ.\ Math. J.},
  FJOURNAL = {Indiana University Mathematics Journal},
    VOLUME = {68},
      YEAR = {2019},
    NUMBER = {4},
     PAGES = {1217--1231},
}



\bib{AESWY}{article}{
  AUTHOR = {Andersson, Mats}, 
  author = {Eriksson, Dennis}, 
  author = {{S}amuelsson {}Kalm, H\aa kan}, 
  author = {Wulcan, Elizabeth}, 
  author = {Yger, Alain},
    TITLE = {Nonproper intersection products and generalized cycles},
  status={accepted for publication in European Journal of Mathematics}
     eprint={arXiv:1812.03054[math.CV]},
   url={http://arxiv.org/abs/arXiv:1812.03054}
 }
 

\bib{ASWY}{article}{
  AUTHOR = {Andersson, Mats}, 
  author = {{S}amuelsson {}Kalm, H\aa kan}, 
  author = {Wulcan, Elizabeth}, 
  author = {Yger, Alain},
     TITLE = {Segre numbers, a generalized {K}ing formula, and local
              intersections},
   JOURNAL = {J.\ Reine Angew.\ Math.},
    VOLUME = {728},
      YEAR = {2017},
     PAGES = {105--136},
}



\bib{AW}{article}{
 AUTHOR = {Andersson, Mats},
 AUTHOR = {Wulcan, Elizabeth},
     TITLE = {Green functions, {S}egre numbers, and {K}ing's formula},
   JOURNAL = {Ann. Inst. Fourier (Grenoble)},
    VOLUME = {64},
      YEAR = {2014},
    NUMBER = {6},
     PAGES = {2639--2657},
}

\bib{BT1}{article}{
   author={Bedford, Eric},
   author={Taylor, B. A.},
   title={The Dirichlet problem for a complex Monge-Amp\`ere equation},
   journal={Invent. Math.},
   volume={37},
   date={1976},
   number={1},
   pages={1--44},
   issn={0020-9910},
}

\bib{BT2}{article}{
   author={Bedford, Eric},
   author={Taylor, B. A.},
   title={A new capacity for plurisubharmonic functions},
   journal={Acta Math.},
   volume={149},
   date={1982},
   number={1-2},
   pages={1--40},
}


\bib{BT3}{article}{
   AUTHOR = {Bedford, Eric}, 
author = {Taylor, B. A.},
     TITLE = {Fine topology, \v{S}ilov boundary, and {$(dd^c)^n$}},
   JOURNAL = {J. Funct. Anal.},
    VOLUME = {72},
      YEAR = {1987},
    NUMBER = {2},
     PAGES = {225--251},
   }

\bib{BB}{article}{
   AUTHOR = {Berman, Robert}, 
   author =  {Boucksom, S\'{e}bastien},
     TITLE = {Growth of balls of holomorphic sections and energy at
              equilibrium},
   JOURNAL = {Invent.\ Math.},
    VOLUME = {181},
      YEAR = {2010},
    NUMBER = {2},
     PAGES = {337--394},
   }


   \bib{B06}{article}{
    AUTHOR = {B{\l}ocki, Zbigniew},
     TITLE = {The domain of definition of the complex {M}onge-{A}mp\`ere
              operator},
   JOURNAL = {Amer. J. Math.},
    VOLUME = {128},
      YEAR = {2006},
    NUMBER = {2},
     PAGES = {519--530},
}
   

\bib{B}{article}{
   author={ {B{\l}ocki}, Zbigniew},
    title={On the complex Monge-Amp\`ere operator for quasi-plurisubharmonic functions with analytic singularities},
    journal={Bull. London Math. Soc.},
volume = {51},
year = {2019},
pages = {431--435},
}


\bib{BEGZ}{article}{
   AUTHOR = {Boucksom, S\'{e}bastien}, 
author = {Eyssidieux, Philippe}, 
author = {Guedj, Vincent}, 
author = {Zeriahi, Ahmed},
     TITLE = {Monge-{A}mp\`ere equations in big cohomology classes},
   JOURNAL = {Acta Math.},
    VOLUME = {205},
      YEAR = {2010},
    NUMBER = {2},
     PAGES = {199--262},
}


\bib{cegrell}{article}{ 
    AUTHOR = {Cegrell, Urban},
     TITLE = {Pluricomplex energy},
   JOURNAL = {Acta Math.},
    VOLUME = {180},
      YEAR = {1998},
    NUMBER = {2},
     PAGES = {187--217},
}


\bib{C}{article}{
    AUTHOR = {Cegrell, Urban},
     TITLE = {The general definition of the complex {M}onge-{A}mp\`ere
              operator},
   JOURNAL = {Ann.\ Inst.\ Fourier (Grenoble)},
    VOLUME = {54},
      YEAR = {2004},
    NUMBER = {1},
     PAGES = {159--179},
}
	

\bib{DDL}{article}{
  AUTHOR = {Darvas, Tam\'{a}s},
  author= {Di Nezza, Eleonora},
  author={Lu, Chinh H.},
     TITLE = {Monotonicity of nonpluripolar products and complex
              {M}onge-{A}mp\`ere equations with prescribed singularity},
   JOURNAL = {Anal. PDE},
    VOLUME = {11},
      YEAR = {2018},
    NUMBER = {8},
     PAGES = {2049--2087},
}
		

\bib{Dem87}{article}{
    AUTHOR = {Demailly, Jean-Pierre},
     TITLE = {Mesures de {M}onge-{A}mp\`ere et mesures pluriharmoniques},
   JOURNAL = {Math.\ Z.},
    VOLUME = {194},
      YEAR = {1987},
    NUMBER = {4},
     PAGES = {519--564},
}

\bib{Dem93}{article}{
    AUTHOR = {Demailly, Jean-Pierre},
     TITLE = {Monge-{A}mp\`ere operators, {L}elong numbers and intersection
              theory},
 BOOKTITLE = {Complex analysis and geometry},
    SERIES = {Univ. Ser. Math.},
     PAGES = {115--193},
 PUBLISHER = {Plenum, New York},
      YEAR = {1993},
}

\bib{D}{book}{
author={{Demailly}, Jean-Pierre}, 
title={Complex Analytic and Differential Geometry}, 
year = {1997}, 
status = {available at
{\tt http://www-fourier.ujf-grenoble.fr/∼demailly/manuscripts/agbook.pdf}}
}


\bib{GZ}{article}{ 
  AUTHOR = {Guedj, Vincent}
  author = {Zeriahi, Ahmed},
     TITLE = {The weighted {M}onge-{A}mp\`ere energy of quasiplurisubharmonic
              functions},
   JOURNAL = {J. Funct. Anal.},
    VOLUME = {250},
      YEAR = {2007},
    NUMBER = {2},
     PAGES = {442--482},
}



\bib{K}{article}{
    AUTHOR = {Kiselman, Christer O.},
     TITLE = {Sur la d\'{e}finition de l'op\'{e}rateur de {M}onge-{A}mp\`ere complexe},
 BOOKTITLE = {Complex analysis ({T}oulouse, 1983)},
    SERIES = {Lecture Notes in Math.},
    VOLUME = {1094},
     PAGES = {139--150},
 PUBLISHER = {Springer, Berlin},
      YEAR = {1984},
}


\bib{LRSW}{article}{
author={{L{\"a}rk{\"a}ng}, Richard},
author={{Raufi}, Hossein},
author={{Sera}, Martin},
author={{Wulcan}, Elizabeth}, 
title={Chern forms of hermitian metrics with analytic singularities on vector bundles},
   JOURNAL = {Indiana Univ. Math. J.},
  FJOURNAL = {Indiana University Mathematics Journal},
    VOLUME = {71},
      YEAR = {2022},
    NUMBER = {1},
     PAGES = {153--189},
      ISSN = {0022-2518},
   MRCLASS = {32U40 (32C30 32W20)},
  MRNUMBER = {4395594},
}
 



\bib{R}{article}{
    AUTHOR = {Rashkovskii, Alexander},
     TITLE = {Analytic approximations of plurisubharmonic singularities},
   JOURNAL = {Math. Z.},
    VOLUME = {275},
      YEAR = {2013},
    NUMBER = {3-4},
     PAGES = {1217--1238},
}


\bib{Duc20}{article}{
  author={Vu, Duc-Viet},
  title={Convexity of the class of currents with finite relative
    energy}, 
   JOURNAL = {Ann.\ Polon.\ Math.},
      VOLUME = {128},
      YEAR = {2022},
    NUMBER = {3},
     PAGES = {275--288}
  }
 

\bib{Duc21}{article}{
  author={Vu, Duc-Viet},
  title={Relative non-pluripolar product of currents}
  journal={Ann.\ Glob.\ Anal.\ Geom.}
     VOLUME = {60},
      YEAR = {2021},
    NUMBER = {2},
     PAGES = {269--311}
    }

\bib{WN}{article}{
author={{Witt Nystr\"om}, David},
title={Monotonicity of non-pluripolar Monge-Amp\`ere masses}
journal={Indiana Univ. Math. J.},
volume={68},
year={2019}
pages={579--591}
}

\bib{Xia19}{article}{
    author={Xia, Mingchen},
    title={Integration by parts formula for non-pluripolar products},
    status={preprint},
     eprint={arXiv: 1907.06359} 
   url={http://arxiv.org/abs/arXiv: 1907.06359}
 } 

 \end{biblist}
\end{bibdiv}
\end{document}